\documentclass[12pt]{amsart}

\usepackage{amssymb,amsthm,amsmath,amscd}

\newtheorem{theorem}{Theorem}[section]
\newtheorem{lemma}[theorem]{Lemma}

\newtheorem{cor}[theorem]{Corollary}
\newtheorem{con}[theorem]{Conjecture}

\theoremstyle{definition}
\newtheorem{definition}[theorem]{Definition}
\newtheorem{example}[theorem]{Example}

\theoremstyle{remark}
\newtheorem{remark}[theorem]{Remark}
\numberwithin{equation}{section}

\newcommand\C{\mathbb{C}}
\newcommand\K{\mathbb{K}}
\newcommand\Z{\mathbb{Z}}

\newcommand\R{\mathbb{R}}
\newcommand\N{\mathbb{N}}
\newcommand\T{\mathbb{T}}

\newcommand\cH{\mathcal{H}}
\newcommand\cK{\mathcal{K}}
\newcommand\cL{\mathcal{L}}

\newcommand\cR{\mathcal{R}}

\newcommand\fA{\mathfrak{A}}
\newcommand\fS{\mathfrak{S}}

\newcommand{\Hom}{\operatorname{Hom}}
\newcommand{\Tr}{\operatorname{Tr}}
\newcommand{\Proj}{\operatorname{Proj}}
\newcommand{\Char}{\operatorname{Char}}
\newcommand{\ex}{\operatorname{ex}}
\newcommand{\tr}{\operatorname{tr}}
\newcommand{\hLambda}{\hat{\Lambda}}

\newcommand\inpr[2]{\langle{#1,#2}\rangle}

\title[Indecomposable characters of infinite dimensional groups]{Indecomposable characters of 
infinite dimensional groups associated with operator algebras} 

\author{Takumi Enomoto}

\address{Department of Mathematics\\ Graduate School of Science\\
Kyoto University\\ Sakyo-ku, Kyoto 606-8502\\ Japan}
\email{takumi.0.enomoto@gmail.com}

\author{Masaki Izumi}
\address{Department of Mathematics\\ Graduate School of Science\\
Kyoto University\\ Sakyo-ku, Kyoto 606-8502\\ Japan}
\email{izumi@math.kyoto-u.ac.jp}


\thanks{Supported in part by the Grant-in-Aid for Scientific Research (B) 22340032, JSPS}

\begin{document} 

\begin{abstract} We determine the indecomposable characters of several classes of infinite dimensional groups associated with 
operator algebras, including the unitary groups of arbitrary unital simple  AF algebras and II$_1$ factors. 
\end{abstract}

\maketitle

\section{Introduction} 
One of the most fundamental tasks in representation theory is to classify irreducible representations of a given group. 
For a locally compact group $G$, this is known to be equivalent to classifying irreducible representations of the group $C^*$-algebra $C^*(G)$. 
The classical Mackey-Glimm dichotomy (see \cite{P79}) says that there is no hope to accomplish it if the group possesses 
a non-type I representation. 
This is a consequence of the fact that if a $C^*$-algebra $A$ has a non-type I representation, the set of pure states $P(A)$ of $A$ 
modulo the natural action of the unitary group $U(A)$ of $A$ does not have a reasonable Borel structure. 
However, we still have a hope to classify finite factor representations because the tracial states of $A$ form a nice Choquet simplex $T(A)$, 
on which $U(A)$ acts trivially. 
A tracial state of the group $C^*$-algebra $C^*(G)$ corresponds to a character of $G$ as we introduce now. 

\begin{definition} Let $G$ be a (not necessarily locally compact) topological group. 
A \textit{character} $\chi$ of $G$ is a positive definite continuous function $\chi: G\to \C$ satisfying 
$\chi(hgh^{-1})=\chi(g)$ for any $g,h\in G$, and $\chi(e)=1$. 
We denote by $\Char(G)$ the set of characters of $G$. 
The character space $\Char(G)$ is a convex set, and an extreme point of $\Char(G)$ is said to be \textit{indecomposable}. 
We denote by $\ex \Char(G)$ the set of indecomposable characters of $G$. 
\end{definition}

The classification of finite factor representations of $G$, up to quasi-equivalence, is equivalent to that of $\ex\Char(G)$, even when 
$G$ is not locally compact (see \cite[Theorem B]{HH05-1}). 

The first classification result of $\ex\Char(G)$ for a non-type I group was obtained by Thoma \cite{Th64}, who gave an explicit description 
of the indecomposable characters of the infinite symmetric group $\fS_\infty$, the inductive limit of the symmetric groups $\fS_n$. 
Thoma's description of $\ex\Char(G)$ involves infinitely many parameters, whose interpretation in terms of infinite paths on 
Yong diagrams was given by Vershik-Kerov \cite{VK81}. 
These works opened up a totally new field in asymptotic representation theory (see \cite{K03}). 

The first classification result for a non-locally compact group was obtained by Voiculescu \cite{V76}, who worked on the infinite unitary group 
$U(\infty)$, the inductive limit of the unitary groups $U(n)$. 
More precisely, he gave a concrete list of indecomposable characters, now called the Voiculescu characters, and its completeness was 
later observed by Vershik-Kerov \cite{VK82} and Boyer \cite{B83} independently (see \cite{OO98} too).  

There are a few more examples of groups whose indecomposable characters are explicitly classified 
(see \cite{S76}, \cite{B93}, \cite{HH05-2}, \cite{HHH08}, \cite{HHH09}, \cite{D}, \cite{DM} for example), but they are restricted to inductive limits of compact groups 
(and their completion in some topology). 
One of the purposes of this paper is to classify the indecomposable characters of a large class of infinite dimensional groups coming from 
operator algebras, and indeed some of them are not in the category of topological groups mentioned above. 
We first classify the indecomposable characters of the unitary groups of simple unital AF algebras. 
Although such an attempt was already done for the CAR algebra by Boyer \cite{B93}, his proof does not seem to be adequate. 
While Boyer studied the structure of the Stratila-Voiculescu AF algebra \cite{SV75} of the inductive limit group 
$U(2^\infty)=\varinjlim U(2^n)$, 
we employ Okounkov-Olshanski's approach in \cite{OO98} based on the Vershik-Kerov ergodic method \cite{VK81}. 

Using the classification result for the unitary groups of unital simple AF algebras, we deduce classification results for broader 
classes of groups. 
For any group $G$ in these classes, the product of two indecomposable characters is again indecomposable, and $\ex\Char(G)$ 
is a multiplicative semigroup. 
For example, we see that every indecomposable character of the unitary group $U(A_\theta)$ of the irrational rotation algebra 
$A_\theta$ is of the form $\psi\tau^p\overline{\tau}^q$, where $\psi$ is a character of the $K_1$-group $K_1(A_\theta)\cong \Z^2$ and 
$\tau$ is the unique trace of $A_\theta$. 
Thus we have a semigroup isomorphism $\ex \Char(U(A_\theta))\cong \T^2\times \Z_{\geq 0}^2$.

To state our main results, we introduce the notation for AF algebras now. 
The reader is referred to \cite{D96} for the basics of AF algebras. 
An AF algebra $A=\varinjlim A_n$ is an inductive limit, in the category of $C^*$-algebras, of an inductive system 
$\{A_n\}_{n=1}^\infty$ of finite dimensional $C^*$-algebras. 
Throughout this note, we assume that the connecting map from $A_n$ to $A_m$ is an embedding for any $n<m$, 
and it is unital whenever $A$ is unital. 
We consider two infinite dimensional groups associated with the AF algebra $A$:  
the full unitary group $U(A)$ of $A$ (if $A$ is not unital, we set $U(A)=\{u\in U(A+\C1);\;u-1\in A \}$) equipped with 
the norm topology, and the inductive limit group $U_{\to}(A)=\varinjlim U(A_n)$ equipped with the inductive limit topology. 
When $A$ is not unital, we embed $U(A_n)$ into $U(A)$ by $u\mapsto u+1-1_{A_n}$, and hence the above inductive limit 
makes sense in $U(A)$. 
The group $U_\to(A)$ is identified with a dense subgroup of the unitary group $U(A)$, 
and the inductive limit topology of $U_\to(A)$ is stronger than the relative topology of 
the norm topology of $U(A)$. 
Note that the isomorphism class of $U_\to(A)$ as a topological group depends only on the 
isomorphism class of $A$, which justifies the notation $U_\to(A)$, 
because Elliott's classification theorem \cite{E} says that isomorphic AF algebras have isomorphic algebraic inductive limits.

The difference between the two groups $U_\to(A)$ and $U(A)$ could be very subtle from the view point of representation theory. 
Indeed, when $A=\K$, the set of compact operators on a separable infinite dimensional Hilbert space, 
the group $U_\to(\K)$ is isomorphic to $U(\infty)$ mentioned above. 
While $U(\infty)$ has uncountably many type II$_1$ factor representations as shown by Voiculescu, Kirillov \cite{K73} showed that $U(\K)$ is 
a type I group with only countably many irreducible representations.  

We introduce the notion of determinant associated with $\varphi\in \Hom(K_0(A),\Z)$. 
For the $K$-theory of $C^*$-algebras, the reader is referred to \cite{Bl98}. 
We denote by $\Proj(A)$ the set of projections in $A$. 
For $u\in U(A_n)$, let $u=\sum_{i=1}^k z_ie_i$ be the spectral decomposition with $z_i\in \T$ 
and $e_i\in \Proj(A_n)$. 
We set $\det_\varphi u=\prod_{i=1}^kz_i^{\varphi([e_i])}$, where $[e_i]\in K_0(A)$ is the $K_0$-class of $e_i$. 
Then $\det_\varphi:U_\to(A)\to \T$ is a well-defined map, which is continuous and multiplicative on $U(A_n)$ for any $n$, 
and hence on $U_\to(A)$.

\begin{theorem} \label{main1} Let $A$ be an infinite dimensional unital simple AF algebra. 
Then \\
$(1)$
\begin{align*}
\lefteqn{\ex\Char(U_\to(A))} \\
 &=\{\mathrm{det}_\varphi(\prod_{i=1}^p\tau_i)(\prod_{j=1}^q\overline{\tau'_j}); \; 
\tau_i,\tau'_j\in \ex T(A),\;p,q\geq 0,\; \varphi\in \Hom(K_0(A),\Z)\}.
\end{align*}
$(2)$
$$\ex\Char(U(A))=\{(\prod_{i=1}^p\tau_i)(\prod_{j=1}^q\overline{\tau'_j}); \; 
\tau_i,\tau'_j\in \ex T(A),\; p,q\geq 0\}.$$
\end{theorem}

It is known that any metrizable Choquet simplex can be realized by $T(A)$ for a unital simple 
AF algebra $A$ (see \cite[Theorem 7.4.3]{Bl98}).

We can generalize Theorem \ref{main1} to a broader class of $C^*$-algebras as follows (see Section 6 for details). 
 
\begin{theorem} \label{main2} 
Let $A$ be a separable infinite dimensional unital simple exact $C^*$-algebra with tracial topological rank 0 
and torsion-free $K_0(A)$. \\
$(1)$ Let $U(A)_0$ be the connected component of $1_A$ in $U(A)$. Then 
$$\ex \Char(U(A)_0)=\{(\prod_{i=1}^p\tau_i)(\prod_{j=1}^q\overline{\tau'_j});\; 
\tau_i,\tau'_j\in \ex T(A),\; p,q\geq 0\}.$$
$(2)$ Let $\widehat{K_1(A)}=\Hom(K_1(A),\T)$ be the dual group of $K_1(A)$. Then 
$$\ex \Char(U(A))=\{\psi(\prod_{i=1}^p\tau_i)(\prod_{j=1}^q\overline{\tau'_j});\; 
\tau_i,\tau'_j\in \ex T(A),\; p,q\geq 0,\; \psi\in \widehat{K_1(A)}\},$$
where we identify $\psi\in \widehat{K_1(A)}$ with the homomorphism 
$U(A)\ni u\mapsto \psi([u])\in \T$. 
\end{theorem}

For a stable AF algebra $A$, we denote by $TW(A)$ the set of densely defined lower semi-continuous 
semifinite traces on $A$. 
All the traces in $TW(A)$ have a common dense domain, called the Pedersen ideal of $A$ (see \cite[Proposition 5.6.7]{P79}), 
and $TW(A)$ is closed under addition and multiplication by positive numbers. 
Since $A_n$ is included in the Pedersen ideal of $A$ for any $n$, 
the function $U_\to(A)\ni u\mapsto \tau(u-1)$ is continuous for any $\tau\in TW(A)$. 
For $\tau,\tau'\in TW(A)$ and $u\in U_\to(A)$, we set 
$\chi_{\tau,\tau'}(u)=e^{\tau(u-1)+\tau'(u^*-1)}$. 
The function $\tau(u-1)+\tau'(u^*-1)$ is conditionally positive definite 
in the following sense: for any complex numbers $c_1,c_2,\ldots,c_n$ with $\sum_{i=1}^nc_i=0$ and 
$u_1,u_2,\ldots,u_n\in U_\to(A)$, we have 
$$\sum_{i,j=1}^n(\tau(u_j^{-1}u_i-1)+\tau'((u_j^{-1}u_i)^*-1))c_i\overline{c_j}=\tau(x^*x)+\tau'(yy^*)\geq 0,$$
where $x=\sum_{i=1}^nc_i(u_i-1)$, $y=\sum_{i=1}^nc_i(u_i^*-1)$. 
Thus  $\chi_{\tau,\tau'}$ is a character of $U_\to(A)$ thanks to the well-known Schoenberg theorem.

\begin{theorem} \label{main3} Let $A$ be a stable simple AF algebra not isomorphic to $\K$. 
If $TW(A)$ is finite dimensional, 
$$\ex \Char(U_\to(A))=\{\mathrm{det}_\varphi\chi_{\tau,\tau'};\; \tau,\tau'\in TW(A),\; 
\varphi\in \Hom(K_0(A),\Z)\}.$$
\end{theorem}

As easy consequences of our main results (and their proofs), we are able to determine the indecomposable characters 
for the unitary groups of arbitrary type II$_1$ factors, and for a family of subgroups of the unitary groups of  
arbitrary type II$_\infty$ factors. 

\begin{theorem} \label{II1} Let $R$ be a type II$_1$ factor with a unique tracial state $\tau$. 
We equip the unitary group $U(R)$ of $R$ with the strong operator topology. 
Then 
$$\ex\Char(U(R))=\{\tau^p\overline{\tau}^q;\; p\geq 0,\; q\geq 0\}.$$ 
Moreover, every $\chi\in \Char(U(R))$ is uniquely decomposed as 
$$\chi=\sum_{p,q\in \Z_{\geq 0}}c_{p,q}\tau^p\overline{\tau}^q,\quad c_{p,q}\geq 0.$$
\end{theorem}

Let $M$ be a type II$_\infty$ factor with separable predual. 
We denote by $\tau_\infty$ the unique (up to scalar multiple) normal semifinite trace of $M$. 
For $1\leq p< \infty$, we set 
$$U(M)_p=\{u\in U(M);\; \|u-1\|_p<\infty \}.$$
where $\|x\|_p=\tau_\infty(|x|^p))^{1/p}$. 
Then $U(M)_p$ is a Polish group with an invariant metric $\|u-v\|_p$ for $u,v\in U(M)_p$. 
For $a,b\geq 0$ and $u\in U(M)_1$, we set $\chi_{a,b}(u)=e^{a\tau_\infty(u-1)+b\tau_\infty(u^*-1)}$. 
For $u\in U(M)_2$, we set $\chi_a(u)=e^{-a\|u-1\|_2^2}$, which was discussed in \cite{AM2012} 
(cf. \cite{B80}). 
For $1\leq q<p$, we have the inclusion relation $U(M)_q\subset U(M)_p$, and 
we have $\chi_{a,a}(u)=\chi_a(u)$ for $u\in U(M)_1$. 

\begin{theorem} \label{IIinfty} Let the notation be as above. 
\begin{itemize}
\item [(1)] $\ex\Char(U(M)_1)=\{\chi_{a,b};\; a,b\geq 0\}.$
\item [(2)] $\ex\Char(U(M)_p)=\{\chi_a;\; a\geq 0\}$ for $1<p\leq 2$. 
\item [(3)] $\ex\Char(U(M)_p)=\{1\}$ for $2\leq p$. 
\end{itemize}
\end{theorem}

In view of Theorem \ref{main1}, \ref{main2}, and \ref{II1}, it would be tempting to conjecture 
that any $\chi\in \ex \Char(U(A))$ is of the form 
$$\chi=\psi(\prod_{i=1}^p\tau_i)(\prod_{j=1}^q\overline{\tau'_j}),\quad 
\tau_i,\tau_j'\in \ex T(A),\; \psi\in \Hom(U(A),\T),$$ 
for any simple unital $C^*$-algebra $A$. 
The first test case beyond the class of $C^*$-algebras discussed here is the Jiang-Su algebra $\mathcal{Z}$ (see \cite{JS}), 
which is left as an open problem. 
Another possible challenge for the future is to determine $\ex\Char(U_\to(A))$ for natural non-simple AF algebras, such as 
the gauge invariant CAR algebra. 
Since  the GICAR has a quotient isomorphic to $\K+\C 1$, the classification list should be considerably complicated.

The authors would like to thank Benoit Collins for informing them of the formula Lemma \ref{IZ},(2). 

\section{Preliminaries}

\subsection{Characters and representations}
A representation $(\pi,\cH)$ of a topological group $G$ consists of a continuous homomorphism $\pi$ 
from $G$ to the unitary group $U(\cH)$ of a Hilbert space $\cH$ equipped with the strong operator topology. 
We often call $\pi$ a representation, and $\cH$ the representation space of $\pi$, which is sometimes denoted by 
$\cH_\pi$. 
A cyclic representation $(\pi, \cH,\Omega)$ of $G$ is a representation $(\pi,\cH)$ with a unit vector 
$\Omega\in \cH$ such that $\pi(G)\Omega$ is a total set in $\cH$.  

For a character $\chi\in \Char(G)$, there exists a unique (up to unitary equivalence) 
cyclic representation $(\pi_\chi,\cH_\chi,\Omega_\chi)$ of $G$ satisfying $\chi(g)=\inpr{\pi(g)\Omega}{\Omega}$. 
We call it the cyclic representation of $G$ associated with $\chi$. 
The von Neumann algebra $M=\pi_\chi(G)''$ is always finite, and $\tau(x)=\inpr{x\Omega_\chi}{\Omega_\chi}$ 
gives a faithful normal tracial state on $M$. 
The von Neumann algebra $M$ is a factor if and only if $\chi\in \ex \Char(G)$. 

If $(\pi,\cH)$ is a representation of $G$ generating a finite von Neumann algebra $M=\pi(G)''$ with 
a faithful normal trace $\tau$, then $\chi(g)=\tau\circ \pi(g)$ gives a character $\chi\in \Char(G)$, and 
$\pi$ is quasi-equivalent to $\pi_\chi$. 

We often use the following well known fact without mentioning it. 

\begin{lemma} Let $G$ be a topological group, let $H$ be a closed subgroup of $G$, and let $\chi\in \Char(G)$. 
Then the restriction $\pi_\chi|_{H}$ of $\pi_\chi$ to $H$ is quasi-equivalent to the cyclic representation 
$\pi_{\chi|_H}$ associated with the restriction $\chi|_H$ of $\chi$ to $H$. 
\end{lemma}

\begin{proof} Let $M=\pi_\chi(G)''$, $N=\pi_\chi(H)''$, $\cK=\overline{N\Omega_\chi}$, and let $p$ be the 
projection from $\cH$ onto $\cK$. 
Then $p\in N'$, and $\pi_{\chi|_H}$ is unitarily equivalent to $(p\pi_\chi|_H,\cK)$. 
Since $\Omega_\chi$ is a separating vector for $M$, we have $\overline{N'\cK}\supset \overline{M'\cK}=\cH_\chi$. 
Thus the map $N\ni x\mapsto px\in pN$ is an isomorphism, which shows the statement. 
\end{proof}

\subsection{Stratila-Voiculescu AF algebras}
Let $G=\varinjlim G_n$ be an inductive limit of second countable compact groups with $G_0=\{e\}$. 
The main purpose of this section is to show that $\ex \Char(G)$ has a Polish topology and every $\omega\in \Char(G)$ 
has an integral expression $\omega=\int_{\ex \Char(G)}\chi d\nu(\chi)$ with a unique Borel probability measure $\nu$ on 
$\ex \Char(G)$. 
This was first proved by Voiculescu \cite{V76} for $U(\infty)$. 
Olshanski \cite{Ol2003} gave a new proof to Voiculescu's result, which in fact works in the general case. 
The main technical problem here is that $\Char(G)$ is not compact in general, and Olshanski gave a 
nice compactification, which is a Choquet simplex.  
We give another proof using the Stratila-Voiculescu AF algebra $\fA(G)$ of $G$, whose tracial simplex $T(\fA(G))$ 
is essentially the same as Olshanski's compactification. 

For the relationship between locally compact groups and $C^*$-algebras, the reader is referred to \cite{P79}. 
The group $C^*$-algebra $C^*(G_n)$ is the universal $C^*$-algebra of the Banach $*$-algebra $L^1(G_n)$. 
Let $\lambda^{(n)}$ be the left regular representation of $G_n$ on $L^2(G_n)$. 
We concretely realize $C^*(G_n)$ as the $C^*$-algebra generated by 
$\lambda^{(n)}(f)=\int_{G_n}f(g)\lambda^{(n)}_gdg$ for $f\in L^1(G_n)$, where $dg$ is the Haar measure of $G_n$. 
The group von Neumann algebra $\cL(G_n)$ is the von Neumann algebra generated by $\lambda^{(n)}_{G_n}$, 
the weak closure of $C^*(G_n)$ too, which coincides with the multiplier algebra of $C^*(G_n)$. 
For a Radon measure $\nu$ on $G_n$, we denote $\lambda^{(n)}(\nu)=\int_{G_n}\lambda^{(n)}d\nu(g)\in \cL(G_n)$. 
Let $\widehat{G_n}$ be the unitary dual of $G_n$. 
Then $C^*(G_n)$ (resp. $\cL(G_n)$) is isomorphic to the direct sum $\oplus_{\pi\in \widehat{G_n}}B(\cH_\pi)$ 
in the category of $C^*$-algebras (resp. von Neumann algebras). 
Since every representation of $C^*(G_n)$ is a direct sum of irreducible representations, it uniquely 
extends to a normal representation of $\cL(G_n)$. 

Since the restriction of $\lambda^{(n)}$ to $G_k$ for $k<n$ is quasi-equivalent to $\lambda^{(k)}$, 
we have a natural embedding $C^*(G_k)\subset \cL(G_k)\subset \cL(G_n)$, which is given by 
$\int_{G_k}f(h)\lambda^{(k)}_hdh \mapsto \int_{G_k}f(h)\lambda^{(n)}_hdh$ for $f\in L^1(G_k)$. 
Thus $\sum_{k=0}^nC^*(G_k)$ makes sense as a $*$-subalgebra of $\cL(G_n)$, and we denote by $\fA_n(G)$ its norm closure. 
We define $\fA(G)$ by the inductive limit $C^*$-algebra $\varinjlim \fA_n(G)$ of the inductive system 
$\{\fA_n(G)\}_{n=0}^\infty$. 
Since $\fA_n(G)$ is an AF algebra for each $n$, so is $\fA(G)$. 

For every representation of $G$, its restriction to $G_n$ induces a normal representation of $\cL(G_n)$, 
which further induces a representation of $\fA(G)$. 
On the other hand, every factor representation $\pi$ of $\fA(G)$ is induced by either a factor representation 
of $G$ as above, or by an irreducible representation of $G_n$, with possibly multiplicity, 
for some $n$ in the following sense. 
Let $J_m$ be the closed ideal generated by $\cup_{k=m+1}^\infty C^*(G_k)$, and let $n$ be the smallest integer 
with $\pi|_{J_n}= 0$. 
Then $\pi$ factors through $\fA(G)/J_n$, which is isomorphic to $\fA_n(G)$,  
and it comes from an irreducible representation of $G_n$ (\cite{SV75}). 

For any character $\chi$ of $G$, we can associate a normal tracial state $\tr_{\chi, n}$ of $\cL(G_n)$ by 
the relation $\tr_{\chi,n}(\lambda^{(n)}(\nu))=\int_{G_n}\chi(g)d\nu(g)$. 
Since $\tr_{\chi,n}$ is normal, the restriction of $\tr_{\chi,n}$ to $C^*(G_n)$ is a state too. 
Then $\tr_{\chi, n}$ is compatibly with the above inductive system, and we obtain a tracial state 
$\tr_{\chi}$ of $\fA(G)$. 
The cyclic representation $(\pi_\chi,\cH_\chi,\Omega_\chi)$ of $G$ associated with $\chi$ and 
the GNS cyclic representation $(\pi_{\tr_\chi},\cH_{\tr_\chi},\Omega_{\tr_\chi})$ of $\fA(G)$ associated with $\tr_\chi$ 
are identified via 
$$\int_{G_n}\pi_\chi(g)f(g)dg=\pi_{\tr_\chi}(\lambda^{(n)}(f)),\quad f\in L^1(G_n).$$

For $x\in \fA(G)$, we denote by $\widehat{x}$ the continuous function on $T(\fA(G))$ defined by 
$\widehat{x}(\tau)=\tau(x)$. 
The topology of $T(\fA(G))$ is the weakest topology making $\widehat{x}$ continuous for any $x\in \fA(G)$. 
We identify $\Char(G)$ with a subset of $T(\fA(G))$ by the correspondence $\chi\mapsto \tr_\chi$, 
and introduce the relative topology (and Borel structure) of $T(\fA(G))$ into $\Char(G)$. 
Then the topology of $\Char(G)$ is the weakest topology making 
$\chi\mapsto \int_{G_n} \chi(g)f(g)dg$ continuous for any $f\in L^1(G_n)$ and any $n$, 
that is, the relative topology of $\sigma(C_b(G),\cup_{n=0}^\infty L^1(G_n))$ (cf. \cite{HH05-1}).  
The inclusion $\Char(G)\subset T(\fA(G))$ provides us a nice compactification of $\Char(G)$, 
which allows us to apply Choquet theory.

We denote by $\fA(G)^{**}$ the second dual of $\fA(G)$, which is known to have a natural von Neumann algebra structure 
(see \cite{P79}). 
For a closed ideal $J$ of $\fA(G)$, we denote by $z_J$ the unit of the weak closure $J''\subset \fA(G)^{**}$ of $J$, 
which is a central projection in $\fA(G)^{**}$. 
In concrete terms, it is obtained as the strong limit in $\fA(G)^{**}$ of an approximate unit $\{u_k\}_{k=1}^\infty$ of $J$.  
We define a lower semicontinuous function $\widehat{z_J}$ on $T(\fA(G))$ by the pointwise limit of 
$\{\widehat{u_k}\}_{k=1}^\infty$, which does not depend on the particular choice of the approximate unit $\{u_k\}_{k=1}^\infty$.

\begin{lemma} Let $G=\varinjlim G_n$ be an inductive limit of second countable compact groups, and let 
$\fA(G)$ be the Stratila-Voiculescu AF algebra for $G$. 
\begin{itemize}
\item[(1)] $\Char(G)$ is a G$_\delta$ subset of $T(\fA(G))$, and in particular it is a Polish space. 
\item[(2)] $\ex \Char(G)=\Char(G)\cap \ex T(\fA(G))$. 
\item[(3)] For any character $\omega\in \Char(G)$, there exists a unique probability measure $\nu$ on $\ex \Char(G)$ 
satisfying 
$$\tr_\omega(x)=\int_{\ex \Char(G)}\tr_\chi(x) d\nu(\chi),\quad \forall x\in \fA(G).$$ 
Moreover, we have 
$$\omega(g)=\int_{\ex \Char(G)}\chi(g)d\nu(\chi),\quad \forall g\in G.$$ 
\end{itemize}
\end{lemma}

\begin{proof} (1) Recall that $J_n$ is the closed ideal of $\fA(G)$ generated by $\cup_{k=n+1}^\infty C^*(G_k)$. 
Let $z_n=z_{J_n}\in \fA(G)^{**}$.
Then we have $0\leq \widehat{z_n}\leq 1$ and $\{\widehat{z_n}\}_{n=0}^\infty$ is a decreasing sequence. 
We claim 
\begin{equation}\label{characterization} 
\{\tr_{\chi}\in T(\fA(G));\; \chi\in \Char(G)\}
=\bigcap_{n=0}^\infty\{\tau \in T(\fA(G));\; \widehat{z_n}(\tau)=1\}.
\end{equation}
Since 
$$\{\tau \in T(\fA(G));\; \widehat{z_n}(\tau)=1\}=\bigcap_{m=1}^\infty
\{\tau \in T(\fA(G));\; \widehat{z_n}(\tau)>1-\frac{1}{m}\},$$
and $\widehat{z_n}$ is lower semicontinuous, the claim would imply that $\Char(G)$ is a G$_\delta$ 
subset of $T(\fA(G))$. 

Assume that $\tau_0\in T(\fA(G))$ does not belong to the right-hand side of (\ref{characterization}), 
and let $n$ be the smallest integer with $\widehat{z_n}(\tau_0)\neq 1$. 
Let $\tau_1(x)=\lim_{k\to\infty} \tau(x-xu_{n,k})/(1-\widehat{z_n}(\tau_0))$, 
where $\{u_{n,k}\}_{k=1}^\infty$ is a quasi-central approximate unit of $J_n$. 
Then $\tau_1$ is in $T(\fA(G))$ that factors through $\fA(G)/J_n$, and 
the GNS representation $\pi_{\tau_1}$ is contained in $\pi_{\tau_0}$. 
Since $\pi_{\tau_1}|_{J_n}=0$, the representation $\pi_{\tau_1}$ does not come from a representation of 
$G$, and neither does $\pi_{\tau_0}$.  
Hence the trace $\tau_0$ does not belong to the left-hand side. 

Assume that $\tau_2\in T(A)$ belongs to the right-hand side of (\ref{characterization}) now. 
Since $T(\fA(G))$ is a Choquet simplex, there exists a unique probability measure $\nu$ on $\ex T(\fA(G))$ 
satisfying $\tau_2=\int_{\ex T(A)}\tau d\nu(\tau)$. 
By monotone convergence theorem, we have 
$$1=\widehat{z_n}(\tau_2)=\int_{\ex T(\fA(G))}\widehat{z_n}(\tau)d\nu(\tau).$$ 
Thus $\nu$ is supported by 
$$C=\bigcap_{n=0}^\infty\{\tau \in \ex T(A);\; \widehat{z_n}(\tau)=1\}.$$
For any $\tau\in \ex T(A)$, the GNS representation $\pi_{\tau}$ comes from either an irreducible representation 
of $G_n$ for some $n$ or a finite factor representation of $G$. 
For $\tau\in C$, the former case does not occur, and there exists a unique character $\chi_{\tau}$ 
satisfying $\tau=\tr_{\chi_{\tau}}$.  
Let $\{f_{n,k}\}_{k=1}^\infty$ be an approximate unit of $C^*(G_n)$, and 
let $\tau_{2,n}$ be the normal extension of $\tau_2|_{C^*(G_n)}$ to $\cL(G_n)$. 
Then by the bounded convergence theorem, we have 
$$\tau_{2,n}(\lambda^{(n)}_g)=\lim_{k\to\infty}\tau_2(\lambda^{(n)}_gf_{n,k})=
\lim_{k\to\infty}\int_{C}\tau(\lambda^{(n)}_gf_{n,k})d\nu(\tau)=
\int_{C}\chi_{\tau}(g)d\nu(\tau).$$ 
We denote by $\chi(g)$ the function on $G$ defined by the last integral. 
The bounded convergence theorem implies that $\chi$ is continuous, and hence it is a character of $G$. 
Now we have $\tr_\chi=\tau_2$.  

(2)
Assume $\chi\in \ex \Char(G)$ and $\tr_\chi=(1-t)\tau_1+t\tau_2$ with $\tau_1,\tau_2\in T(\fA(G))$ and $0<t<1$. 
Then $(1-t)\widehat{z_n}(\tau_1)+t\widehat{z_n}(\tau_2)=1$ for all $n$, and $\widehat{z_n}(\tau_1)=\widehat{z_n}(\tau_2)=1$. 
This shows that $\tau_1$ and $\tau_2$ come from characters of $G$, and $\chi=\tau_1=\tau_2$.

(3) is already shown in the proof of (1). 
\end{proof}

The restriction $\chi|_{G_n}$ of $\chi\in \Char(G)$ to $G_n$ is decomposed as 
$$\chi(g)=\sum_{\pi\in \widehat{G_n}}c_{n,\pi}\frac{\Tr(\pi(g))}{\dim \pi},$$
with non-negative numbers $c_{n,\pi}$ satisfying $\sum_{\pi\in \widehat{G_n}}c_{n,\pi}=1$. 
We set $S_{n,\chi}=\{\pi\in \widehat{G_n};\; c_{n,\pi}\neq 0\}$. 

\begin{lemma}\label{kernel and quotient}
For $\chi\in \Char(G)$, the closed ideal $J_\chi=\ker \pi_{\tr_\chi}$ of $\fA(G)$ depends only on the sets $S_{n,\chi}$, $n\geq 0$. 
If $S_{n,\chi}$ is a finite set for any $n$, then the quotient algebra $\fA(G)/J_\chi$ is identified with the AF algebra 
$$\varinjlim \bigoplus_{\pi\in S_{n,\chi}} B(\cH_\pi).$$
For $\sigma\in S_{n-1,\chi}$ and $\pi\in S_{n,\chi}$, the connecting map from $B(\cH_\sigma)$ to $B(\cH_\pi)$ 
is given by the restriction $\pi|_{G_{n-1}}$. 
\end{lemma}

\begin{proof} Since $\fA(G)$ is an inductive limit $C^*$-algebra, we have 
$J_\chi=\varinjlim (\fA_n(G)\cap J_{\chi})$ and $\fA(G)/J_\chi\cong \varinjlim (\fA_n(G)/(\fA_n(G)\cap J_\chi))$. 
Note that we have $C^*(G_n)\subset \fA_n(G)\subset \cL(G_n)$. 
Since $\pi_{\tr_\chi}|_{\fA_n(G)}$ extends to a normal representation $\pi_{\tr_\chi}^{(n)}$ of $\cL(G_n)$, and 
$$\pi_{\tr_{\chi}}(C^*(G_n))=\pi_{\tr_\chi}^{(n)}(\cL(G_n))\cong \bigoplus_{\pi\in S_{n,\chi}} B(\cH_\pi),$$
we get 
$$\fA_n(G)/(\fA_n(G)\cap J_\chi)\cong \pi_{\tr_{\chi}}(\fA_n(G))\cong \bigoplus_{\pi\in S_{n,\chi}} B(\cH_\pi).$$
\end{proof}

\subsection{Vershik-Kerov ergodic method}
The Vershik-Kerov's ergodic method introduced in \cite{VK81} is a powerful tool to investigate 
$\ex \Char(G)$ for an inductive limit group $G$, and it often provides an intuitive explanation to parameters in 
$\ex \Char(G)$. 
It is an easy consequence of the backward martingale convergence theorem (see \cite{K03}).  

\begin{theorem}[Vershik-Kerov] \label{ergodic method} 
Let $\{G_n\}_{n=0}^\infty$ be an inductive system of topological groups with countable $\ex\Char(G_n)$ for any $n$. 
We assume that any $\omega\in \Char(G_n)$ is uniquely decomposed as 
$$\omega=\sum_{\chi\in \ex \Char(G_n)}c_\chi\chi$$ 
with non-negative $c_\chi$ (and we say that $\omega$ contains $\chi$ if $c_\chi\neq 0$). 
Let $G=\varinjlim G_n$ be the inductive limit group. 
Then for any $\chi\in \ex\Char(G)$, there exist characters $\chi_n\in \ex \Char(G_n)$ such that 
the restriction of $\chi_{n+1}$ to $G_n$ contains $\chi_n$ for any $n$, and 
$\{\chi_n\}_{n=m}^\infty$ converges to $\chi$ uniformly on $G_m$ for any $m$. 
We may further assume that $\chi_n$ is contained in $\chi|_{G_n}$ for any $n$. 
\end{theorem}

Assume that $G_n$ is a second countable compact group for any $n$, and $G_1=U(1)$, e.g. $G=U(\infty)$. 
Let $\{\chi_n\}_{n=0}^\infty$ be a sequence of indecomposable characters with $\chi_n\in \ex\Char(G_n)$ such that 
$\chi_n|_{G_{n-1}}$ contains $\chi_{n-1}$. 
To apply Vershik-Kerov ergodic method to a concrete situation, we have to figure out when 
the sequence $\{\chi_n\}_{n=0}^\infty$ converges to a character of $G$.  

The Fourier expansion $\chi_n(e^{it})=\sum_{k\in \Z}M^{(n)}(k)e^{ikt}$ of $\chi_n$ restricted to $G_1$ gives 
a sequence of finitely supported probability measures $\{M^{(n)}\}_{n=1}^\infty$ on $\Z$. 
Assume that $\{\chi_n\}_{n=0}^\infty$ converges. 
Then $\{M^{(n)}\}_{n=0}^\infty$ is a tight family.  
A crucial observation of Okounkov-Olshanski in \cite{OO98} says that more is true. 
Namely they showed that the second moment sequence is bounded in the case of $U(\infty)$ by using the following easy, 
but nevertheless crucial observation \cite[Lemma 5.2]{OO98}.  
For a natural number $p$, we denote 
$$\langle k^p \rangle_{M^{(n)}}=\sum_{k\in \Z}k^pM^{(n)}(k),$$
if it exists. 

\begin{lemma}\label{OO} Let $\{M^{(n)}\}_{n=1}^\infty$ be a tight family of probability measures on $\Z$ having 4-th moments. 
If the second moment sequence $\{\langle k^2\rangle_{M^{(n)}}\}_{n=1}^\infty$ diverges, then 
the sequence $\{\langle k^4\rangle_{M^{(n)}}/\langle k^2\rangle_{M^{(n)}}^2\}_{n=1}^\infty$ diverges too. 
\end{lemma}

In the case where we have an estimate $\langle k^4\rangle_{M^{(n)}}=O(\langle k^2\rangle_{M^{(n)}}^2)$, the above lemma implies that 
the second moment sequence $\{\langle k^2 \rangle_{M^{(n)}}\}_{n=0}^\infty$ is bounded whenever $\{\chi_n\}_{n=0}^\infty$ converges. 
For $U(\infty)$, Okounkov-Olshanski \cite{OO98} used shifted Schur functions to obtain the estimate.  
For the unitary groups of unital simple AF algebras, it is more convenient to use the Harish-Chandara-Itzykson-Zuber integral instead 
(see Section 3.3). 

\section{The characters of the unitary group $U(d)$}
In this section, we deduce properties of the characters of the unitary group $U(d)$ that are necessary 
for the proof of Theorem \ref{main1}. 
In particular, we give several asymptotic estimates of the characters of $U(d)$. 

For combinatorics associated with the representation theory of $U(d)$, we use the notation and convention in \cite{M95}. 
In particular we don't specify the number of variables for a symmetric homogeneous function $f$. 
We often identify a vector in $\C^d$ with the corresponding diagonal matrix in $M_d(\C)$. 
For a diagonal matrix $A\in M_d(\C)$, we denote 
$f(A)=f(A_{11},A_{22},\ldots,A_{dd})$. 
More generally, if $A\in M_d(\C)$ is a diagonalizable matrix with eigenvalues $\alpha_1,\alpha_2,\ldots,\alpha_d$, 
we denote $f(A)=f(\alpha_1,\alpha_2,\ldots,\alpha_d)$. 

A signature $\Lambda=(\Lambda_1,\Lambda_2,\ldots,\Lambda_d)$ is a tuple of integers satisfying $\Lambda_i\geq \Lambda_{i+1}$ 
for any $1\leq i\leq d-1$. 
It is well known that the set of the equivalence classes $\widehat{U(d)}$ of the irreducible representations of $U(d)$ is naturally 
in one-to-one correspondence with the set of signatures (see \cite{Z}). 
For $\Lambda$, we denote by $(\pi_{\Lambda},\cH_\Lambda)$ the corresponding irreducible representation. 
We recall the Weyl character formula and Weyl dimension formula: 
$$\Tr(\pi_\Lambda(x))=\frac{\det(x_i^{\Lambda_j+d-j})_{i,j}}{\prod_{1\leq i<j\leq d}(x_i-x_j)},\quad x=(x_1,x_2,\ldots,x_d)\in\T^d,$$
$$\dim \pi_\Lambda=\prod_{1\leq i<j\leq d}\frac{\Lambda_i-\Lambda_j+j-i}{j-i}.$$
We set 
$$\chi_{\Lambda}(U)=\frac{\Tr(\pi_\Lambda(U))}{\dim \pi_{\Lambda}}.$$
When there is a possibility of confusion, we use the notation $\pi^{(d)}_{\Lambda}$ and $\chi^{(d)}_\Lambda$ 
instead of $\pi_\Lambda$ and $\chi_\Lambda$. 
In what follows, we identify $\widehat{U(d)}$ with the set of signatures (of length $d$). 
Then we have $\ex\Char(U(d))=\{\chi_{\Lambda}\}_{\Lambda\in \widehat{U(d)}}$. 

A signature $\Lambda$ is characterized by a pair of partitions, or Yong diagrams, as follows. 
We choose $p,q\in \N$ satisfying $\Lambda_p>0\geq \Lambda_{p+1}$ and $\Lambda_{q-1}\geq 0>\Lambda_q$. 
Then $\lambda=(\Lambda_1,\Lambda_2,\ldots,\Lambda_p)$ and $\mu=(-\Lambda_d,-\Lambda_{d-1},\ldots,-\Lambda_q)$ are partitions, 
and $\Lambda$ is uniquely determined by the pair $(\lambda,\mu)$ and $d$. 
Following \cite{King75}, we use the notation $\{\overline{\mu};\lambda\}$ for $\Lambda$ too. 
When the negative part $\mu$ is empty, we have 
$\chi_{\{\overline{\emptyset};\lambda\}}(U)=s_\lambda(U)/s_\lambda(1_d)$ for $U\in U(d)$, where $s_\lambda$ is the Schur polynomial.

\subsection{Branching rules}

We regard $U(d_1)\times U(d_2)$ as a subgroup 
$$\{\left(
\begin{array}{cc}
u &0  \\
0 &v 
\end{array}
\right);\; 
u\in U(d_1),\; v\in U(d_2)\}$$
of $U(d_1+d_2)$. 

\begin{lemma} Let the notation be as above. 
\begin{itemize}
\item [(1)] Let $\{\overline{\mu_i};\lambda_i\}\in \widehat{U(d)}$ for $i=1,2,3$. 
If $\pi_{\{\overline{\mu_3};\lambda_3\}}$ is contained in $\pi_{\{\overline{\mu_1};\lambda_1\}}\otimes \pi_{\{\overline{\mu_2};\lambda_2\}}$, 
we have $|\lambda_3|\leq |\lambda_1|+|\lambda_2|$, $|\mu_3|\leq |\mu_1|+|\mu_2|$, and 
$|\lambda_3|-|\mu_3|=|\lambda_1|+|\lambda_2|-|\mu_1|-|\mu_2|$.
\item [(2)] Let $\{\overline{\mu};\lambda\}\in \widehat{U(d_1+d_2)}$, and let 
$\{\overline{\mu_i};\lambda_i\}\in \widehat{U(d_i)}$ for $i=1,2$. 
If $\pi_{\{\overline{\mu_1};\lambda_1\}}\times \pi_{\{\overline{\mu_2};\lambda_2\}}$ is contained in 
the restriction of $\pi_{\{\overline{\mu};\lambda\}}$ to $U(d_1)\times U(d_2)$, we have 
$|\lambda_1|+|\lambda_2|\leq |\lambda|$, $|\mu_1|+|\mu_2|\leq |\mu|$, and 
$|\lambda_1|+|\lambda_2|-|\mu_1|-|\mu_2|=|\lambda|-|\mu|$.
\end{itemize}
\end{lemma}

\begin{proof} (1) follows from \cite[(3.4),(3.5)]{King75}, and (2) follows from \cite[(5.12)]{King75}. 
\end{proof}

Repeated use of the above lemma implies the following. 

\begin{lemma}\label{branching rule}
Let  
$$A_1=\bigoplus_{j=1}^{N_1}A_{1,j}\subset A_2=\bigoplus_{k=1}^{N_2}A_{2,k}$$
be a unital inclusion of finite dimensional $C^*$-algebras with $A_{1,j}\cong M_{d_{1,j}}(\C)$ and 
$A_{2,k}\cong M_{d_{2,k}}(\C)$. 
We identify $U(A_i)$ with $\prod_{j=1}^{N_j}U(d_{i,j})$ for $i=1,2$.
Let 
$$\pi_i=\pi_{\{\overline{\mu^{(i,1)}};\lambda^{(i,1)}\}}\times 
\pi_{\{\overline{\mu^{(i,2)}};\lambda^{(i,2)}\}}\times \cdots \times 
\pi_{\{\overline{\mu^{(i,N_i)}};\lambda^{(i,N_i)}\}}\in \widehat{U(A_i)}.$$
If the restriction of $\pi_2$ to $U(A_1)$ contains $\pi_1$, then 
$$\sum_{j=1}^{N_1}|\lambda^{(1,j)}|\leq \sum_{k=1}^{N_2}|\lambda^{(2,k)}|,$$
$$\sum_{j=1}^{N_1}|\mu^{(1,j)}|\leq \sum_{k=1}^{N_2}|\mu^{(2,k)}|,$$
$$\sum_{j=1}^{N_1}|\lambda^{(1,j)}|-\sum_{j=1}^{N_1}|\mu^{(1,j)}|= \sum_{k=1}^{N_2}|\lambda^{(2,k)}|-\sum_{k=1}^{N_2}|\mu^{(2,k)}|.$$
\end{lemma}

\subsection{Asymptotic estimate}
We denote by $p_r$ the $r$-th power sum $\sum_i x_i^r$. 

\begin{lemma} Let $\lambda$ be a partition of a natural number $n$. 
Then there exist numbers $c^\lambda_{i_1,i_2,\ldots,i_n}$ depending only on $\lambda$ satisfying 
$$s_\lambda=\sum_{i_1+2i_2+\cdots+ni_n=n}c^\lambda_{i_1,i_2,\ldots,i_n}p_1^{i_1}p_2^{i_2}\cdots p_n^{i_n}.$$
Moreover, 
$$c^{\lambda}_{n,0,\ldots,0}=\prod_{1\leq i<j\leq l(\lambda)}\frac{\lambda_i-\lambda_j+j-i}{j-i}
\prod_{1\leq i\leq l(\lambda)}\frac{(l(\lambda)-i)!}{(l(\lambda)+\lambda_i-i)!}.$$ 
\end{lemma}

\begin{proof} It suffices to show the second statement. 
We consider the case where the number of the variables $d$ is sufficiently large.  
On one hand, we have 
$$s_\lambda(1_d)=c^\lambda_{n,0,\ldots,0}d^n+O(d^{n-1}),\quad (d\to\infty).$$
On the other hand, the Weyl dimension formula implies  
\begin{align*}
\lefteqn{s_\lambda(1_d)=
\prod_{1\leq i<j\leq l(\lambda)}\frac{\lambda_i-\lambda_j+j-i}{j-i}
\prod_{1\leq i\leq l(\lambda)<j\leq d}\frac{\lambda_i+j-i}{j-i}   } \\
&=\prod_{1\leq i<j\leq l(\lambda)}\frac{\lambda_i-\lambda_j+j-i}{j-i}
\prod_{1\leq i\leq l(\lambda)}\frac{(d+\lambda_i-i)!(l(\lambda)-i)!}{(l(\lambda)+\lambda_i-i)!(d-i)!}\\
 &=d^n\prod_{1\leq i<j\leq l(\lambda)}\frac{\lambda_i-\lambda_j+j-i}{j-i}
\prod_{1\leq i\leq l(\lambda)}\frac{(l(\lambda)-i)!}{(l(\lambda)+\lambda_i-i)!}+O(d^{n-1}),
\end{align*}
which shows the statement. 
\end{proof}

\begin{cor}\label{asymp1} For a partition $\lambda$, there exist positive constants $C_\lambda$ and $C'_\lambda$ 
depending only on $\lambda$ such that for any $U\in U(d)$ with $l(\lambda)\leq d$ we have  
$$|s_\lambda(U)-c^\lambda_{|\lambda|,0,\ldots,0}\Tr(U)^{|\lambda|}|\leq C_\lambda d^{|\lambda|-1},$$
$$|\frac{s_\lambda(U)}{s_\lambda(1_d)}-(\frac{1}{d}\Tr(U))^{|\lambda|}|\leq \frac{C'_\lambda}{d}.$$
\end{cor}

\begin{lemma}\label{asymp2} For two partitions $\lambda$ and $\mu$, there exist positive constants $C_{\lambda,\mu}$ and $C'_{\lambda,\mu}$ 
depending only on $\lambda$ and $\mu$ such that for any $U\in U(d)$ we have 
$$|\Tr(\pi_{\{\overline{\mu};\lambda\}}(U))-\overline{s_{\mu}(U)}s_\lambda(U)|\leq C_{\lambda,\mu}d^{|\lambda|+|\mu|-2},$$
$$|\chi_{\{\overline{\mu};\lambda\}}(U)-(\frac{\Tr(U)}{d})^{|\lambda|}(\frac{\overline{\Tr(U)}}{d})^{|\mu|}|\leq \frac{C'_{\lambda,\mu}}{d}.$$
\end{lemma}

\begin{proof} Thanks to \cite[(3.4)]{King75} and the Weyl dimension formula, we see that there exists 
a positive constant $C_{\lambda,\mu}$ satisfying 
$$|\Tr(\pi_{\{\overline{\mu};\lambda\}}(U))-\overline{s_{\mu}(U)}s_\lambda(U)|\leq C_{\lambda,\mu}d^{|\lambda|+|\mu|-2},$$
for any $U\in U(d)$. 
Thus
\begin{align*}
\lefteqn{|\chi_{\{\overline{\mu};\lambda\}}(U)-\frac{\overline{s_\mu(U)}s_\lambda(U)}{s_\mu(1_d)s_\lambda(1_d)}|} \\
 &=\frac{|s_\mu(1_d)s_\lambda(1_d)\Tr(\pi_{\{\overline{\mu};\lambda\}}(U))-\dim \pi_{\{\overline{\mu};\lambda\}}\overline{s_\mu(U)}s_\lambda(U))|}
 {\dim \pi_{\{\overline{\mu};\lambda\}}s_\mu(1_d)s_\lambda(1_d)} \\
 &\leq \frac{|s_\mu(1_d)s_\lambda(1_d)-  \dim \pi_{\{\overline{\mu};\lambda\}}||\Tr(\pi_{\{\overline{\mu};\lambda\}}(U))|}
 {\dim \pi_{\{\overline{\mu};\lambda\}}s_\mu(1_d)s_\lambda(1_d)} \\
 &+\frac{\dim \pi_{\{\overline{\mu};\lambda\}}|\Tr(\pi_{\{\overline{\mu};\lambda\}}(U))-\overline{s_\mu(U)}s_\lambda(U))|}
 {\dim \pi_{\{\overline{\mu};\lambda\}}s_\mu(1_d)s_\lambda(1_d)} \\
 &\leq \frac{2C_{\lambda,\mu}d^{|\lambda|+|\mu|-2}}{s_\mu(1_d)s_\lambda(1_d)}.
\end{align*}
Now the second statement follows from Corollary \ref{asymp1}. 
\end{proof}

\subsection{Moment formula}
Let $\{E_{ij}\}_{1\leq i,j\leq d}$ be the canonical system of matrix units in $M_d(\C)$. 
For an even number $2\leq r\leq d$, we set 
$$F=\sum_{i=1}^{r/2}E_{ii}-\sum_{i=r/2+1}^rE_{ii},$$
Let $\Lambda\in \widehat{U(d)}$ be a signature. 
Since $\chi_\Lambda(e^{\sqrt{-1}tF})$ is a positive definite function on $\T$, the Fourier expansion 
$$\chi_\Lambda(e^{\sqrt{-1}tF})=\sum_{n\in \Z}M(k)e^{\sqrt{-1}kt},$$
gives a probability distribution $\{M(k)\}_{k\in \Z}$ on $\Z$. 
For $p\in \N$, we set
$$m(\Lambda,r,p)=\langle k^p\rangle_M=\sum_{k\in \Z}k^pM(k)=\frac{1}{\sqrt{-1}^p}\frac{d^p}{dt^p}\chi_\Lambda(e^{\sqrt{-1}tF})|_{t=0}.$$
Note that the odd moments vanish as $\chi_{\Lambda}(e^{\sqrt{-1}tF})$ is an even function. 

Our purpose in this section is to obtain the estimate $m(\Lambda,r,4)=O(m(\Lambda,r,2)^2)$ for any 
$r\geq 2d/3$ by using the following Harish-Chandara-Itzykson-Zuber integral.  

\begin{lemma}\label{IZ}
Let $A,B\in M_d(\C)$ be Hermitian matrices with eigenvalues $\{\alpha_i\}_{i=1}^d$ and $\{\beta_i\}_{i=1}^d$ 
respectively, and let $dU$ be the normalized Haar measure of $U(d)$. 
\begin{itemize}
\item[(1)]
 $$\int_{U(d)}e^{\sqrt{-1}\Tr(UAU^{-1}B)}dU
=\frac{\prod_{i=1}^{d-1}i!}{\sqrt{-1}^{d(d-1)/2}}\frac{\det(e^{\sqrt{-1}\alpha_i\beta_j})}{\Delta(A)\Delta(B)}. $$
where 
$$\Delta(A)=\prod_{1\leq i<j\leq d}(\alpha_i-\alpha_j).$$
\item[(2)] Let $n\in \N$. 
$$\int_{U(d)}\Tr(UAU^{-1}B)^ndU=
\sum_{\lambda\vdash n}\frac{\dim \Pi_\lambda s_\lambda(A)s_\lambda(B)}{s_\lambda(1_d)},$$
where $\lambda$ runs over all the partitions of $n$, and $\Pi_\lambda$ is the irreducible representation of 
the symmetric group $\fS_n$ corresponding to the 
partition $\lambda$.
\end{itemize}
\end{lemma}

\begin{proof} (1) follows from \cite[(3.4)]{IZ} and (2) follows from \cite[(3.27)]{IZ}. 
\end{proof}

We set 
$$J(B,r,n)=\int_{U(d)}\Tr(UFU^{-1}B)^ndU
=\sum_{\lambda\vdash n}\frac{\dim \Pi_\lambda s_\lambda(F)s_\lambda(B)}{s_\lambda(1_d)}.$$
Let $\widehat{B}=B-\frac{\Tr(B)}{d}1_d$. 
Since $\Tr(F)=0$, we have $J(B,r,n)=J(\widehat{B},r,n)$. 

\begin{lemma}\label{moment-integral}
Let $\Lambda\in \widehat{U(d)}$ be a signature, and let 
$$\rho_d=((d-1)/2,(d-3)/2,\ldots,-(d-1)/2). $$
Then 
$$m(\Lambda,r,2)=J(\rho_d+\hLambda,r,2)-J(\rho_d,r,2),$$
$$m(\Lambda,r,4)=J(\rho_d+\hLambda,r,4)-6m(\Lambda,r,2)J(\rho_d,r,2)-J(\rho_d,r,4).$$
\end{lemma}

\begin{proof}
Lemma \ref{IZ},(1) and the Weyl character formula (or \cite[Theorem 2]{HC}) imply
$$\chi_{\Lambda}(e^{\sqrt{-1}A})=\prod_{1\leq i<j\leq d}\frac{\frac{\alpha_i-\alpha_j}{2}}{\sin\frac{\alpha_i-\alpha_j}{2}}
\int_{U(d)}e^{\sqrt{-1}\Tr(UAU^{-1}(\Lambda+\rho_d))}dU.$$
When $\Lambda=0$, the left-hand side is $1$. 
Thus we get
$$\chi_\Lambda(e^{\sqrt{-1}tF})\int_{U(d)}e^{\sqrt{-1}t\Tr(UFU^{-1}\rho_d)}dU=\int_{U(d)}e^{\sqrt{-1}t\Tr(UFU^{-1}(\Lambda+\rho_d))}dU.$$
Differentiating the both sides by $t$, we get the statement. 
\end{proof}

Easy computation shows 
$$s_{(2)}=\frac{p_2+p_1^2}{2},\quad s_{(1,1)}=\frac{-p_2+p_1^2}{2},$$
$$s_{(4)}=\frac{p_4}{4}+\frac{p_3p_1}{3}+\frac{p_2^2}{8}+\frac{p_2p_1^2}{4}+\frac{p_1^4}{24},$$
$$s_{(1,1,1,1)}=-\frac{p_4}{4}+\frac{p_3p_1}{3}+\frac{p_2^2}{8}-\frac{p_2p_1^2}{4}+\frac{p_1^4}{24},$$
$$s_{(3,1)}=-\frac{p_4}{4}-\frac{p_2^2}{8}+\frac{p_2p_1^2}{4}+\frac{p_1^4}{8},$$
$$s_{(2,1,1)}=\frac{p_4}{4}-\frac{p_2^2}{8}-\frac{p_2p_1^2}{4}+\frac{p_1^4}{8},$$
$$s_{(2,2)}=-\frac{p_3p_1}{3}+\frac{p_2^2}{4}+\frac{p_1^4}{12},$$
and if $\Tr(B)=0$, we get 
$$s_{(2)}(B)=\frac{\Tr(B^2)}{2},\quad s_{(1,1)}(B)=\frac{-\Tr(B^2)}{2},$$
$$s_{(4)}(B)=\frac{\Tr(B^4)}{4}+\frac{\Tr(B^2)^2}{8},\quad s_{(1,1,1,1)}(B)=-\frac{\Tr(B^4)}{4}+\frac{\Tr(B^2)^2}{8},$$
$$s_{(3,1)}(B)=-\frac{\Tr(B^4)}{4}-\frac{\Tr(B^2)^2}{8},\quad s_{(2,1,1)}(B)=\frac{\Tr(B^4)}{4}-\frac{\Tr(B^2)^2}{8},$$
$$s_{(2,2)}(B)=\frac{\Tr(B^2)^2}{4}.$$
We also have 
$$s_{(2)}(1_d)=\frac{d(d+1)}{2},\quad s_{(1,1)}(1_d)=\frac{d(d-1)}{2},$$
$$s_{(4)}(1_d)=\frac{d(d+1)(d+2)(d+3)}{24},\quad s_{(1,1,1,1)}(1_d)=\frac{d(d-1)(d-2)(d-3)}{24},$$
$$s_{(3,1)}(1_d)=\frac{(d-1)d(d+1)(d+2)}{8},\quad s_{(2,1,1)}(1_d)=\frac{(d-2)(d-1)d(d+1)}{8},$$
$$s_{(2,2)}(1_d)=\frac{(d-1)d^2(d+1)}{12},$$
$\dim \Pi_{(2)}=\dim \Pi_{(1,1)}=1$, $\dim \Pi_{(4)}=\dim\Pi_{(1,1,1,1)}=1$, 
$\dim\Pi_{(2,2)}=2$, $\dim \Pi_{(3,1)}=\dim\Pi_{(2,1,1)}=3$. 
Thus we obtain the following lemma. 

\begin{lemma}\label{2moment} When $\Tr(B)=0$ and $d\geq 4$, we have 
$$J(B,r,2)=\frac{r\Tr(B^2)}{d^2-1},$$
$$m(\Lambda,r,2)=\frac{r\Tr(2\hLambda\rho_d+\hLambda^2)}{d^2-1},$$
\begin{align*}
J(B,r,4)&=3r\frac{(d^4-6d^2+18)r-2d(2d^2-3)}{d^2(d^2-1)(d^2-4)(d^2-9)}\Tr(B^2)^2  \\
 &-6r\frac{(2d^2-3)r-d(d^2+1)}{d(d^2-1)(d^2-4)(d^2-9)}\Tr(B^4). 
\end{align*}
\end{lemma}

We use the following easy lemma for our main estimate Lemma \ref{2-4estimate}. 

\begin{lemma} Let $a,b\in \R^d$ be vectors satisfying $a_i\geq a_{i+1}$ and $b_i\geq b_{i+1}$ for any $1\leq i\leq d-1$. 
If  $\sum_i^na_i=0$, then $\sum_{i=1}^da_ib_i\geq 0$.  
\end{lemma}

\begin{proof} Let $S_0=0$ and $S_i=\sum_{j=1}^ia_j$. 
Then $S_i\geq 0$, $S_d=0$, and 
\begin{align*}
\sum_{i=1}^da_ib_i &=\sum_{i=1}^d(S_i-S_{i-1})b_i=\sum_{i=1}^dS_ib_i-\sum_{i=1}^{d-1}S_ib_{i+1} \\
 &=S_db_d+\sum_{i=1}^{d-1}S_i(b_i-b_{i+1})=\sum_{i=1}^{d-1}S_i(b_i-b_{i+1})\geq 0.
\end{align*}
\end{proof}

In particular, we have $\Tr(\hLambda\rho_d)\geq 0$, $\Tr(\hLambda^3\rho_d)\geq 0$, $\Tr(\hLambda\rho_d^3)\geq 0$. 

\begin{lemma}\label{2-4estimate} 
There exist positive constants $C_1,C_2>0$, independent of $\Lambda$, such that whenever $r\geq 2d/3$ and $d\geq 4$, 
we have $m(\Lambda,r,4)\leq C_1m(\Lambda,r,2)^2+C_2m(\Lambda,r,2)$. 
\end{lemma}

\begin{proof} Lemma \ref{moment-integral} and Lemma \ref{2moment} imply 
\begin{align*}
\lefteqn{m(\Lambda,r,4)=J(\rho_d+\hLambda,r,4)-J(\rho_d,r,4)-6m(\Lambda,r,2)J(\rho_d,r,2)} \\
 &=3r\frac{(d^4-6d^2+18)r-2d(2d^2-3)}{d^2(d^2-1)(d^2-4)(d^2-9)}(\Tr((\hLambda+\rho_d)^2)^2-\Tr(\rho_d^2)^2) \\
 &-6r\frac{(2d^2-3)r-d(d^2+1)}{d(d^2-1)(d^2-4)(d^2-9)}\Tr((\hLambda+\rho_d)^4-\Tr(\rho_d^4))-6m(\Lambda,r,2)\frac{r\Tr(\rho_d^2)}{d^2-1}\\
 &\leq 3r\frac{(d^4-6d^2+18)r}{d^2(d^2-1)(d^2-4)(d^2-9)}(\Tr(2\hLambda\rho_d+\hLambda^2)^2+2\Tr(2\hLambda\rho_d+\hLambda^2)\Tr(\rho_d^2)) \\
 &-6r\frac{(2d^2-3)r-d(d^2+1)}{d(d^2-1)(d^2-4)(d^2-9)}\Tr(\hLambda^4+4\hLambda^3\rho_d+6\hLambda^2\rho_d^2+4\hLambda\rho_d^3)\\
 &-6m(\Lambda,r,2)\frac{r\Tr(\rho_d^2)}{d^2-1}.
\end{align*}
Since the second term is negative, we have  
\begin{align*}
\lefteqn{m(\Lambda,r,4) \leq 3r^2\frac{(d^4-6d^2+18)}{d^2(d^2-1)(d^2-4)(d^2-9)}} \\
 &\times (\frac{(d^2-1)^2m(\Lambda,r,2)^2}{r^2}+\frac{2(d^2-1)m(\Lambda,r,2)\Tr(\rho_d^2)}{r}) -6m(\Lambda,r,2)\frac{r\Tr(\rho_d^2)}{d^2-1}\\
 &\leq 3m(\Lambda,r,2)^2\frac{(d^2-1)(d^4-6d^2+18)}{d^2(d^2-4)(d^2-9)}\\
 &+36m(\Lambda,r,2)\Tr(\rho_d^2)r\frac{d^4-2d^2-3}{d^2(d^2-1)(d^2-4)(d^2-9)}.\\
\end{align*}
Since $\Tr(\rho_d^2)=d(d^2-1)/12$, we get 
\begin{align*}
\lefteqn{m(\Lambda,r,4)}\\
 &\leq 3m(\Lambda,r,2)^2\frac{(d^2-1)(d^4-6d^2+18)}{d^2(d^2-4)(d^2-9)}
+3m(\Lambda,r,2)\frac{r(d^4-2d^2-3)}{d(d^2-4)(d^2-9)}\\
 &\leq 3m(\Lambda,r,2)^2\frac{(d^2-1)(d^4-6d^2+18)}{d^2(d^2-4)(d^2-9)}
+2m(\Lambda,r,2)\frac{d^4-2d^2-3}{(d^2-4)(d^2-9)},
\end{align*}
which shows the statement. 
\end{proof}

\section{Schur-Weyl duality for the hyperfinite II$_1$ factor}

For a tracial state $\tau$ on a $C^*$-algebra or a von Neumann algebra, we denote 
$\|x\|_\tau=\tau(x^*x)^{1/2}$. 
When $\tau$ is unique, we denote $\|x\|_2=\|x\|_\tau$. 
The following is an analogue of the Schur-Weyl duality theorem (cf. \cite{K73}). 

\begin{theorem} \label{Schur-Weyl1} 
Let $\cR_0$ be the hyperfinite II$_1$ factor acting on $L^2(\cR_0)$ with a cyclic and separating 
trace vector $\Omega\in L^2(\cR_0)$, and let $J$ be the canonical conjugation defined by $Jx\Omega=x^*\Omega$ 
for $x\in \cR_0$. 
For non-negative integers $p,q$ with $(p,q)\neq (0,0)$, 
$$\{u^{\otimes p}\otimes (JuJ)^{\otimes q}\in \cR_0^{\otimes p}\otimes \cR_0'^{\otimes q};\; u\in U(\cR_0)\}''=
(\cR_0^{\otimes p})^{\fS_p}\otimes (\cR_0'^{\otimes q})^{\fS_q},$$
where the symmetric group $\fS_p$ (resp. $\fS_q$) acts on $\cR_0^{\otimes p}$ (resp. $\cR_0'^{\otimes q}$) 
as the permutations of tensor components. 
In particular, the above von Neumann algebra is isomorphic to the hyperfinite II$_1$ factor. 
\end{theorem}

\begin{proof} Since $\cR_0$ is hyperfinite, there exists an increasing sequence of finite dimensional von Neumann subalgebras 
$\{A_n\}_{n=1}^\infty$ whose union is dense in $\cR_0$. 
We may further assume that $A_n=M_{2^n}(\C)$. 
We set  
$$P=\{u^{\otimes p}\otimes (JuJ)^{\otimes q}\in \cR_0^{\otimes p}\otimes \cR_0'^{\otimes q};\; u\in U(\cR_0)\}''$$
$$P_n=\{u^{\otimes p}\otimes (JuJ)^{\otimes q}\in \cR_0^{\otimes p}\otimes \cR_0'^{\otimes q};\; u\in U(A_n)\}'',$$
$$Q=(\cR_0^{\otimes p})^{\fS_p}\otimes (\cR_0'^{\otimes q})^{\fS_q},$$
$$Q_n=(A_n^{\otimes p})^{\fS_p}\otimes ((JA_nJ)^{\otimes q}))^{\fS_q}.$$
Then $\cup_{n=1}^\infty P_n$ is dense in $P$ and $\cup_{n=1}^\infty Q_n$ is dense in $Q$. 
We have an obvious inclusion $P_n\subset Q_n$. 
To prove $P=Q$, it suffices to show the following statement: there exists a decreasing sequence of positive numbers $\{a_n\}_{n=1}^\infty$ 
converging to 0 such that for any $x$ in the unit ball of $Q_n$ there exists $y$ in the unit ball of $P_n$ satisfying $\|x-y\|_2\leq a_n$. 

Let $H_n=\C^{2^n}$, and let $\overline{H_n}$ be its complex conjugate Hilbert space. 
We identify $A_n^{\otimes p}\otimes (JA_nJ)^{\otimes q}$ with $B(H_n^{\otimes p})\otimes B(\overline{H_n}^{\otimes q})$. 
Let $(\Pi_\lambda,\cK_\lambda)$ be the irreducible representation of $\fS_p$ corresponding to the partition $\lambda$ of $p$. 
Then we can make the following identification: 
$$H_n^{\otimes p}\otimes \overline{H_n}^{\otimes q}=\bigoplus_{\lambda\vdash p,\;\mu\vdash q} 
\cH_\lambda\otimes\overline{\cH_\mu}\otimes \cK_\lambda\otimes \cK_\mu,$$
$$Q_n=\bigoplus_{\lambda\vdash p,\;\mu\vdash q} 
B(\cH_\lambda\otimes\overline{\cH_\mu})\otimes 1_{\cK_\lambda}\otimes 1_{\cK_\mu},$$
$$P_n=\{\bigoplus_{\lambda\vdash p,\;\mu\vdash q}
(\pi_\lambda\otimes\overline{\pi_\mu})(U)\otimes  1_{\cK_\lambda}\otimes 1_{\cK_\mu};
\; U\in U(2^n)\}''. $$

For a pair of partitions $(\lambda,\mu)$ with $|\lambda|=p$ and $|\mu|=q$,  
the irreducible representation $\pi_{\{\overline{\mu};\lambda\}}$ of $U(2^n)$ is contained in $\pi_\lambda\otimes \overline{\pi_\mu}$ 
with multiplicity 1, and it is not contained in  $\pi_{\lambda'}\otimes\overline{\pi_{\mu'}}$ with a different pair $(\lambda',\mu')$. 
Let $e_{\lambda,\mu}\in B(\cH_\lambda\otimes \overline{\cH_\mu})$ be the projection onto $\cH_{\{\overline{\mu};\lambda\}}$, 
and let  
$$e_n=\bigoplus_{\lambda\vdash p,\;\mu\vdash q} e_{\lambda,\mu}\otimes 1_{\cK_\lambda}\otimes 1_{\cK_\mu}.$$
Then $e_nQ_ne_n\subset P_n$. 
For $x$ in the unit ball $Q_n$, we have 
$$\|x-e_nxe_n\|_2\leq \|x(1-e_n)+(1-e_n)xe_n\|_2\leq 2\|x\|\|1-e_n\|_\tau\leq 2\|1-e_n\|_2.$$
Thus it suffices to show that the sequence $\{\|1-e_n\|_2\}_{n=1}^\infty$ converges to 0. 
Indeed, since 
\begin{align*}
\lefteqn{\|1-e_n\|_2^2=\frac{\Tr(1-e_n)}{2^{n(p+q)}}} \\
 &=\frac{1}{2^{n(p+q)}}
\sum_{\lambda\vdash p,\;\mu\vdash q}(s_\lambda(1_{2^n})s_\mu(1_{2^n})-\dim\pi_{\{\overline{\mu};\lambda\}})
\dim\cK_\lambda\dim\cK_\mu,
\end{align*}
the convergence follows from Lemma \ref{asymp2}.
\end{proof}

\begin{remark} Although we don't know if the statement of the above theorem holds for an arbitrary 
II$_1$ factor $R$, we see from Theorem \ref{II1} that the von Neumann algebra 
$$\{u^{\otimes p}\otimes (JuJ)^{\otimes q}\in R^{\otimes p}\otimes R'^{\otimes q};\; u\in U(R)\}''$$
is always a II$_1$ factor. 
\end{remark}

\begin{theorem} \label{Schur-Weyl2} 
Let $A$ be a unital $C^*$-algebra, and let $\tau_1,\tau_2,\cdots \tau_m\in \ex T(A)$ 
be distinct extreme traces. 
Let $(\pi_i,\cH_i,\Omega_i)$ be the GNS triple for $\tau_i$, and let $J_i$ be the canonical conjugation 
defined by $J_i\pi_i(x)\Omega_i=\pi_i(x)^*\Omega_i$. 
We assume that $R_i=\pi_i(A)''$ is isomorphic to the hyperfinite II$_1$ factor for every $1\leq i\leq m$. 
Let $(p_i,q_i)\in \Z_{\geq 0}\times \Z_{\geq 0}\setminus (0,0)$, and let 
$$\pi(u)=\bigotimes_{i=1}^m 
(\pi_i(u)^{\otimes p_i}\otimes (J_i\pi_i(u)J_i)^{\otimes q_i}).$$ 
Then we have
$$\{\pi(u);\; u\in U(A)\}''=\bigotimes_{i=1}^m 
((R_i^{\otimes p_i})^{\fS_{p_i}}\otimes (R_i'^{\otimes q_i})^{\fS_{q_i}}).$$
In particular, we have $\prod_{i=1}^m\tau_i^{p_i}\overline{\tau_i}^{q_i}\in \ex\Char(U(A))$.   
\end{theorem}

\begin{proof} Let $z_i\in A^{**}$ be the central cover of the representation $\pi_i$ (see \cite[3.8.1]{P79}), and let $u\in U(A)$. 
Thanks to the Kaplansky density theorem, there exists a net $\{u_\alpha\}_{\alpha\in \Lambda}$ in $U(A)$ 
converging to $z_iu+1-z_i$ in $U(A^{**})$ in the strong operator topology. 
Then we have the following strong limit: 
$$\lim_\alpha\pi_j(u_\alpha)^{\otimes p_j}\otimes (J_j\pi_j(u_\alpha)J_j)^{\otimes q_j}=
\left\{
\begin{array}{ll}
\pi_j(u)^{\otimes p_j}\otimes (J_j\pi_j(u)J_j)^{\otimes q_j} , &\quad j=i  \\
1 , &\quad j\neq i
\end{array}
\right.
.$$
Thus the statement follows from Theorem \ref{Schur-Weyl1}. 
\end{proof}

\section{Unital simple AF algebras}
Throughout this section, we fix an infinite dimensional unital simple AF algebra $A=\varinjlim A_n$, 
$A_n=\oplus_{i=1}^{N_n}A_{n,i}$ with $A_{n,i}\cong M_{d_{n,i}}(\C)$.  
We assume that $A_n\subset A_{n+1}$ is a unital embedding. 
Let $G_n=U(A_n)$ and $G_{n,i}=U(A_{n,i})\cong U(d_{n,i})$.  
Then $G_n=G_{n,1}\times G_{n,2}\times \cdots \times G_{n,N_n}$, and $U_\to(A)$ is the inductive limit of $\{G_n\}_{n=1}^\infty$ 
We may assume $G_0=\{e\}$.  

It is easy to show the following two lemmas. 

\begin{lemma}\label{dimension} Let $m_n=\inf_{1\leq i\leq N_n}d_{n,i}$. Then $\lim_{n\to\infty}m_n=\infty$. 
\end{lemma}

\begin{lemma} Let $\{\delta_n\}_{n=1}^\infty$ be a sequence of continuous homomorphisms $\delta_n:G_n\to \T$ such that 
the restriction of $\delta_{n+1}$ to $G_n$ is $\delta_n$ for any $n\in \N$. 
Then there exists $\varphi\in\Hom(K_0(A),\Z)$ satisfying $\delta_n=\det_\varphi|_{G_n}$. 
\end{lemma}

\begin{proof}[Proof of Theorem \ref{main1},(1)] 
Thanks to Theorem \ref{Schur-Weyl2}, we already know that the character 
$\mathrm{det}_\varphi(\prod_{i=1}^p\tau_i)(\prod_{j=1}^q\overline{\tau'_j})$ is indecomposable 
for any $\varphi\in \Hom(K_0(A),\Z)$ and $\tau_i,\tau'_j\in \ex T(A)$  

Let $\chi\in \ex\Char(U_{\to}(A))$. 
Then Theorem \ref{ergodic method} shows that there exist indecomposable characters $\chi_n\in \ex\Char(G_n)$ for $n\geq 0$ 
such that $\chi_n|_{G_{n-1}}$ contains $\chi_{n-1}$ and 
$$\chi(U)=\lim_{n\to\infty}\chi_n(U),\quad U\in U_\to(A),$$
where convergence is uniform on $G_m$ for every $m$. 
We denote by $Q_{n,i}$ the projection of $A_n$ onto $A_{n,i}$.  
Then there exists $\chi_{n,i}\in \ex\Char(G_{n,i})$ such that $\chi_n(U)=\prod_{i=1}^{N_n}\chi_{n,i}(Q_{n,i}(U))$. 
Identifying $G_{n,i}$ with $U(d_{n,i})$, we see that there exist signatures $\Lambda^{(n,i)}\in \widehat{U(d_{n,i})}$ with 
$\chi_{n,i}=\chi_{\Lambda^{(n,i)}}$. 

Thanks to Lemma \ref{dimension}, we see that $A_n$ has a unital copy of either $M_2(\C)$, $M_3(\C)$ or $M_2(\C)\oplus M_3(\C)$ 
for large $n$. 
Thus we may assume that $A_1$ is of one of the above forms. 
If $A_1=M_2(\C)$, we set 
$$F=\left(
\begin{array}{cc}
1 &0  \\
0 &-1 
\end{array}
\right)
\in A_1,$$ 
if $A_1=M_3(\C)$, we set 
$$F=\left(
\begin{array}{ccc}
1 &0 &0  \\
0 &-1 &0  \\
0 &0 &0 
\end{array}
\right)\in A_1,$$ 
and if $A_1=M_2(\C)\oplus M_3(\C)$, we set  
$$F=\left(
\begin{array}{cc}
1 &0  \\
0 &-1 
\end{array}
\right)\oplus 
\left(
\begin{array}{ccc}
1 &0 &0  \\
0 &-1 &0  \\
0 &0 &0 
\end{array}
\right)\in A_1.
$$ 
In any case, we have $\Tr(Q_{n,i}(F^2))\geq 2d_{n,i}/3$, and the conclusion of Lemma \ref{2-4estimate} applies to 
$\chi_{\Lambda_{n,i}}(e^{\sqrt{-1}t Q_{n,i}(F)})$.

As in Section 3.3, we define the $p$-th moments of the probability distributions on $\Z$ given by the Fourier expansions of 
$\chi_n(e^{\sqrt{-1}tF})$ and $\chi_{n,i}(e^{\sqrt{-1}tQ_{n,i}(F)})$ by 
$$m(n,p)=\frac{1}{\sqrt{-1}^p}\frac{d^p}{dt^p}\chi_n(e^{\sqrt{-1}tF})|_{t=0},$$
$$m(n,i,p)=\frac{1}{\sqrt{-1}^p}\frac{d^p}{dt^p}\chi_{n,i}(e^{\sqrt{-1}tQ_{n,i}(F)})|_{t=0}.$$
Since $\chi_{n,i}(e^{\sqrt{-1}tQ_{n,i}(F)})$ is an even function, we have 
$$m(n,2)=\sum_{i=1}^{N_n}m(n,i,2),$$
$$m(n,4)=\sum_{i=1}^{N_n}m(n,i,4)+6\sum_{1\leq i<j\leq N_n}m(2,i,2)m(2,j,2).$$
Thus Lemma \ref{2-4estimate} and Lemma \ref{dimension} show that when $n$ is sufficiently large, we have 
\begin{align*}
m(n,4)&\leq C_1\sum_{i=1}^{N_n}m(n,i,2)^2+C_2\sum_{i=1}^{N_n}m(n,i,2)+3m(n,2)^2\\
 &\leq (C_1+3)m(n,2)^2+C_2m(n,2). 
\end{align*}
Now Lemma \ref{OO} implies that $\{m(n,2)\}_{n=1}^\infty$ is bounded, say $m(n,2)\leq C$ for all $n\geq 1$, 
and Lemma \ref{2moment} implies that there exists $n_0\geq 1$ such such for any $n\geq n_0$, we have  
$$\sum_{i=1}^{N_n}\frac{\Tr(\Lambda^{(n,i)}\rho_{d_{n,i}})}{d_{n,i}}
=\sum_{i=1}^{N_n}\frac{\Tr(\widehat{\Lambda^{(n,i)}}\rho_{d_{n,i}})}{d_{n,i}}\leq C.$$

Let $\Lambda^{(n,i)}=(\Lambda^{(n,i)}_1,\Lambda^{(n,i)}_2,\ldots,\Lambda^{(n,i)}_{d_{n,i}})$. 
When $\Lambda^{(n,i)}_1=\Lambda^{(n,i)}_{d_{n,i}}$, we set $l_{n,i}=0$, and otherwise 
we set $l_{n,i}=\max\{j;\; \Lambda^{(n,i)}_j>\Lambda^{(n,i)}_{d_{n,i}+1-j}\}$. 
Then we have $l_{n,i}\leq d_{n,i}/2$. 
For $n\geq n_0$, 
\begin{align*}
C&\geq \sum_{i=1}^{N_n}\frac{\Tr(\Lambda^{(n,i)}\rho_{d_{n,i}})}{d_{n,i}} 
 =\sum_{i=1}^{N_n}\frac{1}{d_{n,i}}\sum_{j=1}^{l_{n,i}}(\frac{d_{n,i}+1}{2}-j)(\Lambda^{(n,i)}_j-\Lambda^{(n,i)}_{d_{n,i}+1-j})\\
 &\geq \sum_{i=1}^{N_n}\frac{1}{d_{n,i}}\sum_{j=1}^{l_{n,i}}(\frac{d_{n,i}+1}{2}-j)
 \geq \sum_{i=1}^{N_n}\frac{(d_{n,i}-l_{n,i})l_{n,i}}{2d_{n,i}}
 \geq \frac{1}{4}\sum_{i=1}^{N_n}l_{n,i}.
\end{align*}

We choose $n_1\geq n_0$ so that for any $n\geq n_1$ and $i$ we have $16C<d_{n,i}$. 
Assume $n\geq n_1$. 
Then $l_{n,i}<d_{n,i}/4$ for all $i$. 
Let $a_{n,i}=\Lambda^{(n,i)}_{l_{n,i}+1}$. 
Then we have $\Lambda^{(n,i)}_j=a_{n,i}$ for any $l_{n,i}+1\leq j\leq d_{n,i}-l_{n,i}$, 
and there exist partitions $\lambda^{(n,i)}$ and $\mu^{(n,i)}$ with $l_{n,i}=\max\{l(\lambda^{(n,i)}),l(\mu^{(n,i)})\}$ satisfying 
$$\chi_{n,i}(U)=(\det U)^{a_{n,i}}\chi_{\{\overline{\mu^{(n,i)}};\lambda^{(n,i)}\}}(U),\quad U\in U(A_{n,i}).$$ 
Setting 
$$\delta_n(U)=\prod_{i=1}^{N_n}(\det Q_{n,i}(U))^{a_{n,i}},\quad U\in U(A_n)$$
$$\chi'_n(U)=\prod_{i=1}^{N_n} \chi_{\{\overline{\mu^{(n,i)}};\lambda^{(n,i)}\}}(Q_{n,i}(U)),\quad U\in U(A_n).$$
we get $\chi_n=\delta_n\chi'_n$. 
We have 
\begin{align*}
C&\geq \sum_{i=1}^{N_n}\frac{1}{d_{n,i}}\sum_{j=1}^{l_{n,i}}(\frac{d_{n,i}+1}{2}-j)(\lambda^{(n,i)}_j+\mu^{(n,i)}_j)\\
 &\geq \sum_{i=1}^{N_n}\frac{1}{d_{n,i}}\sum_{j=1}^{l_{n,i}}(\frac{d_{n,i}+1}{2}-l_{n,i})(\lambda^{(n,i)}_j+\mu^{(n,i)}_j)
 \geq \frac{1}{4} \sum_{i=1}^{N_n}\sum_{j=1}^{l_{n,i}}(\lambda^{(n,i)}_j+\mu^{(n,i)}_j)\\
 &\geq \frac{1}{4}\sum_{i=1}^{N_n}(|\lambda^{(n,i)}|+|\mu^{(n,i)}|).
\end{align*}
We set $p_n=\sum_{i=1}^{N_n}|\lambda^{(n,i)}|$ and $q_n=\sum_{i=1}^{N_n}|\mu^{(n,i)}|$, which are bounded by $4C$.
 
We claim that the restriction $\delta_{n+1}|_{G_n}$ of $\delta_{n+1}$ to $G_n$ coincides with $\delta_n$ for any $n\geq n_1$, 
and in consequence the restriction $\chi'_{n+1}|_{G_n}$ contains $\chi'_n$. 
Indeed, the restriction $\chi'_{n+1}|_{G_n}$ contains $(\delta_{n+1}|_{G_n})^{-1}\delta_n\chi'_n$ as  
$\chi_{n+1}|_{G_n}$ contains $\chi_n$. 
On the other hand, we have $p_n+q_n\leq 4C$ and $p_{n+1}+q_{n+1}\leq 4C$. 
Now the claim follows from Lemma \ref{branching rule} and $16C<d_{n,i}$. 

Since $\delta_{n+1}|G_n=\delta_n$ for all $n\geq n_1$, there exists $\varphi\in \Hom(K_0(A),\Z)$ satisfying 
$\delta_n(U)=\det_\varphi U$ for any $U\in G_n$ with $n\geq n_1$. 
Lemma \ref{branching rule} implies that two sequences $\{p_n\}_{n=n_1}^\infty$ and 
$\{q_n\}_{n=n_1}^\infty$ are bounded and increasing, and hence eventually constants. 
Thus there exist non-negative integers $p,q$, and $n_2\geq n_1$ satisfying $p_n=p$ and $q_n=q$ for any $n\geq n_2$. 

We assume $n\geq n_2$ and set $I_n=\{i;\;|\lambda^{(n,i)}|+|\mu^{(n,i)}|\neq 0\}$. 
Then $\#I_n\leq 4C$. 
Let $C'_{\lambda,\mu}$ be as in Lemma \ref{asymp2}, and let $C'=\max\{C'_{\lambda,\mu};\;|\lambda|\leq p,\; |\mu|\leq q\}$. 
Let $\tau_{n,i}$ be the normalized trace of $A_{n,i}$. 
Then Lemma \ref{asymp2} implies 
$$|\chi'_n(U)-\prod_{i\in I_n}\tau_{n,i}(Q_{n,i}(U))^{|\lambda^{(n,i)}|}\overline{\tau_{n,i}(Q_{n,i}(U))}^{|\mu^{(n,i)}|}|\leq \frac{4CC'}{m_n},$$
for any $U\in G_n$. 
Thus for any $U\in U_\to(A)$, we have 
$$\chi(U)=\det{}_\varphi U\lim_{n\to\infty}\prod_{i\in I_n}\tau_{n,i}(Q_{n,i}(U))^{|\lambda^{(n,i)}|}\overline{\tau_{n,i}(Q_{n,i}(U))}^{|\mu^{(n,i)}|}.$$
Note that we can extend $\tau_{n,i}\circ Q_{n,i}$ to a state of $A$. 
Thus there exist states $\omega_{n,j}\in S(A)$, $\omega'_{n,k}\in S(A)$ for $1\leq j\leq p$, $1\leq k\leq q$ whose 
restrictions to $A_n$ are tracial states such that 
$$\chi(U)=\det{}_\varphi U\lim_{n\to\infty}\prod_{j=1}^p\omega_{n,i}(U)\prod_{k=1}^q
\overline{\omega'_{n,k}(U)}.$$
Since any cluster points of $\{\omega_{n,j}\}_{n=n_2}^\infty$ and $\{\omega'_{n,k}\}_{n=n_2}^\infty$ in the weak* topology 
are tracial states, there exist $\tau_j, \tau'_k\in T(A)$ for $1\leq j\leq p$, $1\leq k\leq q$ satisfying 
$$\chi(U)=\det{}_\varphi U\prod_{j=1}^p\tau_j(U)\prod_{k=1}^q
\overline{\tau'_k(U)}.$$
Since $\chi$ is indecomposable, we conclude $\tau_j,\tau'_k\in \ex T(A)$. 
\end{proof}

\begin{lemma}\label{continuity1} Let $\varphi\in \Hom (K_0(A),\Z)$. 
Then $\det_\varphi$ extends to a continuous homomorphism from $U(A)$ to $\T$ if and only if $\varphi=0$. 
\end{lemma}

\begin{proof} For a projection $e\in \cup_{n=1}^\infty A_n$ and $z\in \T$, we have 
$\det_\varphi (ze+1_A-e)=z^{\varphi([e])}$. 
Thus if $\det_\varphi$ continuously extends to $U(A)$, the set $\{\varphi([e]);\;e\in \Proj(A)\}$ is bounded. 
Thus $m_n>\sup\{|\varphi([e])|;\;e\in \Proj(A)\}$ for sufficiently large $n$. 
Since $\varphi([1_{A_{n,i}}])$ is a multiple of $d_{n,i}$, we have $\varphi([e])=0$ for any $e\in \Proj(A_n)$ 
for sufficiently large $n$, which shows $\varphi=0$. 
\end{proof}

\begin{proof}[Proof of Theorem \ref{main1},(2)] Since $U_\to(A)$ is dense in $U(A)$ and the embedding map 
of $U_\to(A)$ into $U(A)$ is continuous, the statement follows from (1) and Lemma \ref{continuity1}. 
\end{proof}

\begin{example}
Let $A$ be the CAR algebra. 
Then $U_\to(A)$ is the group $U(2^\infty)$ discussed in \cite{B93}. 
In this case $T(A)$ is a singleton $\{\tau\}$, and $\Hom(K_0(A),\Z)=\{0\}$ as $K_0(A)\cong\Z[1/2]$. 
Thus 
$$\ex \Char(U(2^\infty))=\ex\Char U(A)=\{\tau^p\overline{\tau}^q;\; p,q\in \Z_{\geq 0}\}.$$
Our argument also works for the inductive limit group $\varinjlim SU(2^n)$ with the inclusion map 
$$SU(2^n)\ni U\mapsto \left(
\begin{array}{cc}
U &0  \\
0 &U 
\end{array}
\right)\in SU(2^{n+1}),
$$
and the conclusion is the same. 
\end{example}

\begin{example} For any irrational number $0<\theta<1$, there exists a unital simple AF algebra $B_\theta$ with 
$$(K_0(B_\theta),K_0(B_\theta)_+,[1])\cong (\Z+\theta \Z,(\Z+\theta\Z)\cap [0,\infty),1),$$
(see \cite[VI.3]{D96}). 
In this case $T(B_\theta)$ is a singleton $\{\tau\}$, and $\Hom(K_0(B_\theta),\Z)\cong \Z^2$. 
Thus we get $\ex\Char(U_\to(B_\theta))\cong \Z^2\times \Z_{\geq 0}^2$ and $\ex\Char(U(B_\theta))\cong \Z_{\geq 0}^2$. 
\end{example}

In the rest of this section, we discuss integral decomposition of characters in $\Char(U(A))$. 

\begin{lemma} Let $p,q$ be non-negative integers, and let 
$\varphi\in \Hom(K_0(A),\Z)$. 
We set 
$$\ex_{\varphi}^{p,q}\Char(U_\to(A))=\{\mathrm{det}_\varphi(\prod_{i=1}^p\tau_i)(\prod_{j=1}^q\overline{\tau'_j})
\in \Char(U_\to(A)); \; \tau_i,\tau'_j\in \ex T(A)\}.$$ 
$(1)$
Let $\chi\in \ex_{\varphi}^{p,q}\Char(U_\to(A))$. 
An indecomposable character of $G_n$ is contained in $\chi|_{G_n}$ if and only if it is of the form 
$\det_\varphi\prod_{i=1}^n\chi_{\{\overline{\mu^{(n,i)}},\lambda^{(n,i)}\}}\circ Q_{n,i}$ with 
$$\sum_{i=1}^{N_n}|\lambda^{(n,i)}|\leq p,\quad \sum_{i=1}^{N_n}|\mu^{(n,i)}|\leq q,$$
$$\sum_{i=1}^{N_n}|\lambda^{(n,i)}|-\sum_{i=1}^{N_n}|\mu^{(n,i)}|=p-q.$$
$(2)$ $\ex_{\varphi}^{p,q}\Char(U_\to(A))$ is a Borel subset of $\ex\Char(U_\to(A))$. 
\end{lemma}

\begin{proof} (1) The statement follows from \cite[(3.4)]{King75} and the fact that any trace in $T(A)$ is faithful as $A$ is simple. 

(2) Pick $\chi\in \ex_\varphi^{p,q}(\Char(U_\to(A)))$, and set $J_\varphi^{p,q}=\ker \pi_{\tr_\chi}$, 
which is a primitive ideal of $\fA(U_\to(A))$. 
(1) and Lemma \ref{kernel and quotient} show that $J_\varphi^{p,q}$ does not depend on the choice of $\chi$. 
We claim 
\begin{equation}\label{Borel}
\bigcup_{i=0}^{\min\{p,q\}}\ex_\varphi^{p-i,q-i}\Char(U_\to(A))=\{\chi\in \ex\Char(U_\to(A));\; 
\widehat{z_{J_\varphi^{p,q}}}(\tau_\chi)=0\},
\end{equation}
which would imply the statement. 

Let $\omega\in \ex_\varphi^{p-i,q-i}\Char(U_\to(A))$. 
(1) implies the inclusion relation $J_\varphi^{p,q}\subset J_\varphi^{p-i,q-i}$, and hence 
$0\leq \widehat{z_\varphi^{p,q}}\leq \widehat{z_\varphi^{p-i,q-i}}\leq 1$. 
Since $\widehat{z_\varphi^{p-i,q-i}}(\tr_\omega)=0$, we get $\widehat{z_\varphi^{p,q}}(\tr_\omega)=0$, 
and $\omega$ belongs to the right-hand side of (\ref{Borel}). 

Assume that $\omega$ belongs to the right-hand side of (\ref{Borel}) conversely. 
Then the trace $\tau_\omega$ factors through the primitive quotient $\fA(U_\to(A))/J_\varphi^{p,q}$ whose structure 
is described in Lemma \ref{kernel and quotient}. 
Thus the Vershik-Kerov ergodic method gives a sequence of indecomposable characters $\omega_n\in \ex \Char(G_n)$ converging $\omega$ 
such that each $\omega_n$ is as in (1). 
Now the proof of Theorem \ref{main1},(1) shows $\omega\in \bigcup_{i=0}^{\min\{p,q\}}\ex_\varphi^{p-i,q-i}\Char(U_\to(A))$. 
\end{proof}

We set 
$$\ex_0\Char(U_\to(A))=\bigcup_{p,q\geq 0}\ex_0^{p,q}\Char(U_\to(A)).$$

\begin{lemma}\label{continuity2} A character $\omega\in \Char(U_\to(A))$ extends to a character of $U(A)$ 
if and only if there exists a probability measure $\nu$ on $\ex_0\Char(U_\to(A))$ satisfying 
$$\omega=\int_{\ex_0\Char(U_\to(A))}\chi d\nu(\chi).$$
\end{lemma}

\begin{proof} It is easy to see that $\omega\in \Char(U_\to(A))$ with a measure $\nu$ as in the statement 
continuously extends to $U(A)$. 

Let $\omega$ be a character of $U_\to(A)$ with the integral decomposition 
$$\omega=\int_{\ex \Char(U_\to(A))} \chi d\nu(\chi)=\sum_{\varphi\in \Hom(K_0(A),\Z)}\sum_{p,q\geq 0}
\int_{\ex_\varphi^{p,q}\Char(U_\to(A))}\chi d\nu(\chi),$$
with $\nu(\ex_\varphi^{p,q}\Char(U_\to(A)))>0$ for some $\varphi\in \Hom(K_0(A),\Z)\setminus \{0\}$ and $p,q\geq 0$. 
For any projection $e$ in $\cup_{n=1}^\infty A_n$ and $z\in \C$ with $|z|=1$, we have 
$$1-\Re \omega(ze+1-e)\geq \int_{\ex_\varphi^{p,q}\Char(U_\to(A))}(1-\Re\chi(ze+1-e)) d\nu(\chi).$$ 
We assume that $\omega$ extends continuously to $U(A)$, and deduce contradiction. 

Since $\omega$ is norm continuous, for any $0<\varepsilon<\nu(\ex_0^{p,q}\Char(U_\to(A)))$, 
there exists $\delta>$ such that whenever $z\in \T$ satisfies $|z-1|<\delta$, we have 
$1-\Re \omega(ze+1-e)<\varepsilon$. 
For any $\chi\in \ex_\varphi^{p,q}\Char(U_\to(A))$, there exist $\tau_i,\tau'_j\in \ex T(A)$ satisfying 
$$\chi(ze+1-e)=z^{\varphi([e])}\prod_{i=1}^p\tau_i(ze+1-e)\prod_{j=1}^q\overline{\tau'_j(ze+1-e)}.$$
Thus there exists $0<\delta_1<\delta$, depending only on $p$ and $q$, such that if $|z-1|<\delta_1$, 
then 
$$\Re \prod_{i=1}^p\tau_i(ze+1-e)\prod_{j=1}^q\overline{\tau'_j(ze+1-e)}>1/2.$$
Since $\varphi$ is not trivial, we can find $e$ and $z$ satisfying $|z-1|<\delta_1$ and 
$z^{\varphi([e])}=-1$. 
For such $e$ and $z$, we have 
$$\varepsilon>1-\Re \omega(ze+1-e)\geq \int_{\ex_\varphi^{p,q}\Char(U_\to(A))}d\nu(\chi)=\nu(\ex_\varphi^{p,q}\Char(U_\to(A))),$$
which is contradiction. 
\end{proof}

Since $U_\to(A)$ is dense in $U(A)$ and the embedding $U_\to(A)\subset U(A)$ is continuous, 
we identify $\chi\in \Char(U(A))$ with its restriction to $U_\to(A)$. 
Thanks to Lemma \ref{continuity2}, we can identify $\ex_0\Char(U_\to(A))$ with $\ex \Char(U(A))$, and 
regard the latter as a Borel space.  

\begin{cor} \label{decomposition} 
For any character $\omega\in \Char(U(A))$, there exists a unique probability measure 
$\nu$ on $\ex \Char(U(A))$ satisfying $\omega(g)=\int_{\ex \Char(U(A))}\chi(g)d\nu(\chi)$ 
for any $g\in U(A)$. 
\end{cor}

\begin{proof}
Lemma \ref{continuity2} shows that for any $\omega\in \Char(U(A))$, 
there exists a unique probability measure $\nu$ on $\ex \Char(U(A))$ 
satisfying $\omega(g)=\int_{\ex \Char(U(A))}\chi(g)d\nu(\chi)$ for any $g\in U_\to(A)$. 
Now the statement follows from the bounded convergence theorem. 
\end{proof}

For later use, we set 
$$\ex^{p,q}\Char(U(A))=\{(\prod_{i=1}^p\tau_i)(\prod_{j=1}^q\overline{\tau'_j}); \; 
\tau_i,\tau'_j\in \ex T(A)\},$$
which is a Borel subset of $\ex\Char(U(A))$. 
\section{Unital simple $C^*$-algebras of tracial topological rank 0}

Throughout this section, we fix a separable infinite dimensional unital simple exact $C^*$-algebra $A$ 
of tracial topological rank 0. 
The reader is referred to \cite{L01} for the definition of tracial topological rank and the basics of 
this class of $C^*$-algebras. 
We recall that $A$ has real rank 0, stable rank 1, and that $K_0(A)$ is weakly unperforated. 
In consequence, we can identify $U(A)/U(A)_0$ with $K_1(A)$, and two projections in $A$ have the 
same $K_0$-class if and only if they are Murray-von Neumann equivalent. 
Since $A$ is exact, every quasi-trace on $A$ is a trace (see \cite{HT99} for the proof),  
and the trace space $T(A)$ is identified with the state space of 
the scaled ordered group $(K_0(A),K_0(A)_+,[1])$  (see \cite[Theorem 6.9.1]{Bl98}). 

When $K_0(A)$ is torsion-free, it is unperforated, and it is a simple dimension group 
(see \cite[Proposition 3.3.7, Theorem 3.3.18]{L01}, \cite[Theorem IV.7.2]{D96}). 
Thus there exists a unital embedding of a unital simple AF algebra $B$ into $A$ that 
induces an isomorphism of the $K_0$-groups. 
The restriction map from $T(A)$ to $T(B)$ is a bijection.

\begin{lemma} \label{density} 
For any $u\in U(A)$, there exists a sequence $\{u_n\}_{n=1}^\infty$ in $U(A)_0$ such that 
the sequence $\{\|u-u_n\|_\tau\}_{n=1}^\infty$ converges to 0 for any $\tau\in T(A)$.  
\end{lemma}

\begin{proof} We can choose a decreasing sequence of non-zero projections $\{e_n\}_{n=0}^\infty$ in $A$ 
satisfying $e_0=1$, and $[e_{n}]\geq 2[e_{n+1}]$ in $K_0(A)$ for any $n$. 
Indeed, this is possible because $e_nAe_n$ has tracial topological rank 0 too and has a subalgebra isomorphic to $M_2(\C)$.  
Since $A$ has tracial topological rank 0, there exist a projection $p_n\in A$, a 
finite dimensional $C^*$-subalgebra $C_n\subset p_nAp_n$ with $p_n\in C_n$, and $v_n\in C_n$ satisfying 
$[1-p_n]\leq [e_n]$, $\|[u,p_n]\|<1/n$, and  $\|p_nup_n-v_n\|<1/n$. 
We get a desired $u_n$ by perturbing $1-p_n+v_n$.  
\end{proof}

Lemma \ref{density} shows that for any finite factor representation $(\pi,\cH)$ of $A$, the 
sequence $\{\pi(u_n)\}_{n=1}^\infty$ converges to $\pi(u)$ in the strong operator topology. 
Since $A$ is tracial AF, the factor $\pi(A)''$ is hyperfinite. 
Thus Lemma \ref{Schur-Weyl2} shows the following. 

\begin{cor} Let $\tau_i,\tau'_j\in \ex T(A)$ for $1\leq i\leq p$, $1\leq j\leq q$. 
Then 
$$(\prod_{i=1}^p\tau_i)(\prod_{j=1}^q\overline{\tau'_j})\in \ex \Char(U(A)_0).$$ 
\end{cor}

\begin{lemma}\label{invaraince} 
Assume that $K_0(A)$ is torsion free. 
Then for any $\omega\in \Char(U(A)_0)$, we have $\omega(vuv^*)=\omega(u)$ for any $u\in U(A)_0$ and $v\in U(A)$. 
\end{lemma}

\begin{proof} Let $\omega\in \Char(U(A)_0)$. 
We fix a unital embedding of a simple AF algebra $B$ into 
$A$ that induces an isomorphism on the $K_0$-groups. 
By \cite[Theorem 4.2.8]{L01}, for any $\varepsilon >0$ and any $u\in U(A)_0$, there exists 
$u'\in U(A)_0$ with finite spectrum satisfying $\|u-u'\|<\varepsilon$. 
Let $u'=\sum_{i=1}^nz_ip_i$ be the spectral decomposition of $u'$ with $z_i\in \T$ and $p_i\in \Proj(A)$. 
Since $K_0(A)=K_0(B)$, there exists a partition of unity $1=\sum_{i=1}^nq_i$ in $B$ consisting of projections satisfying $[p_i]=[q_i]$, 
and there exists $w\in U(A)$ satisfying $wp_iw^*=q_i$ for all $i$. 
Choosing a unitary $v\in U(q_1Aq_1)$ with $[v+1-q_1]=-[w]$ and replacing $w$ with $(v+1-q_1)w$ if necessary, we may further assume 
$w\in U(A)_0$.  
Therefore there exist $w_n\in U(A)_0$ and $u_n\in U(B)$ such that $\{\|w_nuw_n^*-u_n\|\}_{n=1}^\infty$ 
converges to 0, and so $\omega(u)=\lim_{n\to\infty}\omega(u_n)$.

Since the restriction map from $T(A)$ to $T(B)$ is an affine isomorphism, any $\tau\in T(B)$ has 
a unique extension $\tilde{\tau}\in T(A)$. 
Let 
$$\omega|_{U(B)}=\int_{\ex \Char(U(B))}\chi d\nu(\chi)=\sum_{p,q\geq 0}\int_{\ex^{p,q}\Char(U(B))}
\chi d\nu(\chi)$$ be the integral decomposition of $\omega|_{U(B)}$. 
For any $\chi\in \ex^{p,q}(\Char(U(B)))$, there exist 
$\tau_i,\tau'_j\in \ex T(B)$ for $1\leq i\leq p$, $1\leq j\leq q$ satisfying 
$\chi=(\prod_{i=1}^p\tau_i)(\prod_{j=1}^q\overline{\tau'_j})$. 
We set $\tilde{\chi}=(\prod_{i=1}^p\tilde{\tau_i})(\prod_{j=1}^q\overline{\tilde{\tau'_j}})$, 
whose restriction to $U(A)_0$ is in $\ex \Char(U(A)_0)$. 
By bounded convergence theorem, we have 
$$\omega(u)=\lim_{n\to\infty}\omega(u_n)=\lim_{n\to\infty}\int_{\ex\Char(U(B))}\tilde{\chi}(u_n)d\nu(\chi)
=\int_{\ex\Char(U(B))}\tilde{\chi}(u)d\nu(\chi).$$
Since $\tilde{\chi}(vuv^*)=\tilde{\chi}(u)$ holds for any $u\in U(A)_0$ and $v\in U(A)$, 
we get $\omega(vuv^*)=\omega(u)$. 
\end{proof}

In the above proof, the sequence $\{u_n\}_{n=1}^\infty$ depends only on $u$, and it does not 
depend on $\omega$. 
This means that $\omega$ is completely determined by $\omega|_{U(B)}$, and the integral expression 
$$\omega(u)=\int_{\ex\Char(U(B))}\tilde{\chi}(u)d\nu(\chi)$$
is unique. 
Therefore we get 

\begin{cor}\label{component} 
Assume that $K_0(A)$ is torsion free. 
Then 
$$\ex \Char(U(A)_0)=\{(\prod_{i=1}^p\tau_i)(\prod_{j=1}^q\overline{\tau'_j});\; 
\tau_i,\tau'_j\in \ex T(A),\; p,q\geq 0\}.$$
\end{cor}

\begin{proof}[Proof of Theorem \ref{main2}] 
Let $\chi\in \ex \Char(U(A))$, and let $(\pi,\cH,\Omega)$ be the cyclic representation 
of $U(A)$ associated with $\chi$. 
Let $M=\pi(U(A))''$, and let $\tr(x)=\inpr{x\Omega}{\Omega}$ for $x\in M$. 
Then $M$ is a finite factor and $\tr$ is a trace on $M$. 
We first claim that $N=\pi(U(A)_0)''$ is a factor too. 

Let $z\in Z(N)$ be a non-zero central projection. 
For $u\in U(A)_0$, we set $\omega(u)=\tr(z\pi(u))/\tr(z)$. 
Then $\omega\in \Char(U(A)_0)$, and we have $\omega(vuv^{-1})=\omega(u)$ 
for any $u\in U(A)_0$ and $v\in U(A)$ thanks to Lemma \ref{invaraince}. 
This implies 
$$\tr(z\pi(u))=\tr(z\pi(vuv^{-1}))=\tr(\pi(v^{-1})z\pi(v)\pi(u)),$$
and $\tr((z-\pi(v^{-1})z\pi(v))\pi(u))=0$. 
Note that $z-\pi(v^{-1})z\pi(v)\in N$ as $v$ normalizes $U(A)_0$, and hence $N$. 
Since $\mathrm{span}\pi(U(A)_0)$ is dense in $N$ in the weak operator topology 
and $\tr$ is faithful on $N$, we get $z=\pi(v^{-1})z\pi(v)$ for any $v\in U(A)$. 
This implies that $z$ is in the center of $M$, and so $z=1$. 
Thus $N$ is a factor. 

Since $\pi|_{U(A)_0}$ is a finite factor representation, we have $\chi|_{U(A)_0}\in \ex \Char(U(A)_0)$, and 
$\pi|_{U(A)_0}$ is quasi-equivalent to the cyclic representation associated with 
$\chi|_{U(A)_0}\in \ex \Char(U(A)_0)$. 
Assume $N=\C$. 
Then $\chi|_{U(A)_0}$ is multiplicative, and Corollary \ref{component} implies that $\chi|_{U(A)_0}$ is trivial. 
Thus $\chi$ comes from 
$$\ex \Char(U(A)/U(A)_0)=\ex \Char(K_1(A))=\widehat{K_1(A)}.$$ 

Assume $N\neq \C$. 
Then Corollary \ref{component} implies that there exist distinct extreme traces 
$\tau_1,\tau_2,\cdots, \tau_m\in \ex T(A)$ and non-negative integers $p_i,q_i$ with $\sum_{i=1}^m(p_i+q_i)\neq 0$ 
satisfying $\chi|_{U(A)_0}=\prod_{i=1}^m\tau_i^{p_i}\overline{\tau_i}^{q_i}|_{U(A)_0}$. 
Let $(\pi_i,\cH_i,\Omega_i)$ be the GNS triple for $\tau_i$, and let $J_i$ be the canonical conjugation 
given by $J_i\pi_i(x)\Omega_i=\pi_i(x)^*\Omega_i$. 
For $u\in U(A)$, we set 
$$\sigma(u)=\bigotimes_{i=1}^m \pi_i(u)^{\otimes p_i}\otimes (J_i\pi_i(u)J_i)^{\otimes q_i}.$$
Then $\pi|_{U(A)_0}$ and  $\sigma|_{U(A)_0}$ are quasi-equivalent. 
Lemma \ref{density} implies that $\sigma(U(A)_0)''=\sigma(U(A))''$, which we denote by $P$. 
Then there exists an isomorphism $\theta$ from $P$ onto $N$ satisfying $\theta\circ \sigma(u)=\pi(u)$ 
for any $u\in U(A)_0$. 

For $v\in U(A)$, we set $\gamma(v)=\theta\circ \sigma(v^{-1})\pi(v)$. 
We claim that $\gamma$ is a homomorphism from $U(A)$ to $U(N'\cap M)$. 
Indeed, since $U(A)_0$ is a normal subgroup of $U(A)$,  we have the following 
for any $u\in U(A)_0$ and $v\in U(A)$: 
\begin{align*}
\lefteqn{\gamma(v)\pi(u)\gamma(v)^{-1}=\theta\circ\sigma(v^{-1})\pi(vuv^{-1})\theta\circ\sigma(v)}\\
 &=\theta\circ\sigma(v^{-1})\theta\circ \sigma(vuv^{-1})\theta\circ\sigma(v)
 =\theta\circ \sigma(u)=\pi(u),
\end{align*}
which shows that $\gamma(v)\in N'\cap M$. 
Let $v,w\in U(A)$. 
Then we have $\gamma(v)\theta\circ\sigma(w^{-1})=\theta\circ \sigma(w^{-1})\gamma(v)$ 
as $\gamma(v)\in N'\cap M$ and $\theta\circ \sigma(w)\in N$. 
Thus 
$$\gamma(v)\gamma(w)=\gamma(v)\theta\circ\sigma(w^{-1})\pi(w)=\theta\circ\sigma(w^{-1})\gamma(v)\pi(w)
=\gamma(vw),$$
which shows that $\gamma$ is a homomorphism. 

Since $\gamma(u)=1$ for any $u\in U(A)_0$, the homomorphism $\gamma:U(A)\rightarrow N'\cap M$ 
factors through $U(A)/U(A)_0=K_1(A)$, and $\gamma(U(A))''$ is commutative. 
Since $M$ is a factor and it is generated by $N$ and $\gamma(U(A))$, we obtain $\gamma(U(A))\subset \T$. 
Thus there exists $\psi\in \widehat{K_1(A)}$ satisfying $\gamma(v)=\psi([v])$ for any $v\in U(A)$. 
Now we have 
$$\chi(v)=\tr(\pi(v))=\tr(\theta\circ \sigma(v)\gamma(v))=\psi([v])\tr(\theta\circ \sigma(v)).$$
Since $\tr\circ \theta$ is a unique trace on $P$, we conclude
$$\chi(v)=\psi([v])\prod_{i=1}^m\tau_i^{p_i}(v)\overline{\tau_i(v)^{q_i}}.$$ 
\end{proof}

\begin{example} Let $0<\theta<1$ be an irrational number. 
The irrational rotation algebra $A_\theta$ is the universal $C^*$-algebra generated by two unitaries $u,v$ satisfying 
the relation $uv=e^{2\pi\sqrt{-1}\theta}vu$ (see \cite[Chapter VI]{D96}). 
It is a separable nuclear simple $C^*$-algebra with a unique tracial state and with $K$-theoretical data   
$$(K_0(A),K_0(A)_+,[1],K_1(A))\cong (\Z+\theta\Z,(\Z+\theta\Z)\cap [0,\infty),1,\Z^2).$$
Elliott-Evans \cite{EE} showed that $A_\theta$ is an AT algebra of real rank 0, and  
hence it has tracial topological rank 0 (see \cite[Theorem 4.3.5]{L01}). 
Thus $\ex\Char(U(A_\theta)_0)\cong \Z_{\geq 0}^2$ and $\ex\Char(U(A_\theta))\cong \T^2\times \Z_{\geq 0}^2$. 
One might wonder if $U(A_\theta)_0$ is isomorphic to the unitary group $U(B)$ of an AF algebra $B$, but this is never the case.  
Indeed, on one hand the homotopy group $\pi_2(U(B))$ is trivial for any AF algebra $B$ 
because $\pi_2(U(d))$ is trivial for any $d\geq 1$. 
On the other hand we have $\pi_2(U(A_\theta)_0)\cong K_1(A_\theta)\cong \Z^2$ (see \cite[Theorem II]{Z91}). 
\end{example}

\section{Stable simple AF algebras}
The purpose of this section is to prove Theorem \ref{main3}. 

\begin{lemma}
$$\lim_{t\to +\infty}e^{-t}\sum_{n=1}^\infty |\frac{t^{n-1}}{(n-1)!}-\frac{t^n}{n!}|=0.$$
\end{lemma}

\begin{proof} Let $[t]$ be the integer part of $t>0$. Then 
\begin{align*}
\sum_{n=1}^\infty |\frac{t^{n-1}}{(n-1)!}-\frac{t^n}{n!}|
&=-\sum_{n=1}^{[t]}(\frac{t^{n-1}}{(n-1)!}-\frac{t^n}{n!})
+\sum_{n=[t]+1}^\infty (\frac{t^{n-1}}{(n-1)!}-\frac{t^n}{n!})\\
&=-1+\frac{2t^{[t]}}{[t]!}. 
\end{align*}
Letting $a_t=t-[t]$, we get the following estimate from the Stirling formula:
\begin{align*}
\frac{e^{-t}t^{[t]}}{[t]!}
&=\frac{e^{-[t]-a_t}([t]+a_t)^{[t]}}{[t]!} 
\sim \frac{e^{-[t]-a_t}([t]+a_t)^{[t]}}{[t]^{[t]}e^{-[t]}\sqrt{2\pi[t]}}
=\frac{e^{-a_t}(1+\frac{a_t}{[t]})^{[t]}}{\sqrt{2\pi[t]}}
\\
&\sim\frac{1}{\sqrt{2\pi t}}.
\end{align*}
\end{proof}

\begin{lemma} \label{tail boundary} Let $a=(a_1,a_2,\ldots,a_m)\in \R_{>0}^m$. 
For $x,y\in \Z_{\geq 0}^m$, we introduce a transition probability from $x$ to $y$ by 
$$p_a(x,y)=\left\{
\begin{array}{ll}
\prod_{i=1}^m \frac{e^{-a_i}a_i^{y_i-x_i}}{(y_i-x_i)!} , &\quad y_i-x_i\geq 0,\; \forall i \\
0 , &\quad \mathrm{otherwise}
\end{array}
\right.,
$$
and a Markov operator $P$ acting on $\ell^\infty(\Z_{\geq 0}^m)$ by 
$$P(f)(x)=\sum_{y\in \Z_{\geq 0}^m}p_a(x,y)f(y).$$
Then the tail boundary for $P$ is trivial, that is, 
if $\{f_n\}_{n=1}^\infty$ is a bounded sequence in $\ell^\infty(\Z_{\geq 0}^m)$ satisfying 
$P(f_{n+1})=f_n$ for any $n\in \N$, then there exists a constant $c$ satisfying $f_n=c$ for any $n\in \N$. 
\end{lemma}

\begin{proof} Since $\{\frac{e^{-a}a^x}{x!}\}_{x=0}^\infty$ is the Poisson distribution, the $k$-step 
transition probability $p^{(k)}_a(x,y)$ is given by $p_{ka}(x,y)$. 
We denote by $\{e_i\}_{i=1}^m$ the canonical basis of $\R^m$. 
Let $C=\sup_{n\in \N}\|f_n\|_\infty$. 
Then
\begin{align*}
\lefteqn{|f_n(x)-f_n(x+e_i)|=|P^k(f_{n+k})(x)-P^k(f_{n+k})(x+e_i)|} \\
&\leq \sum_{y\in \Z_{\geq 0}^m}|p_{ka}(x,y)-p_{ka}(x+e_i,y)||f_{n+k}(y)|\\
&\leq C\sum_{z\in \Z_{\geq 0}^m}|p_{ka}(x,x+z)-p_{ka}(x+e_i,x+z)|\\
&=C\sum_{z\in \Z_{\geq 0}^m}|p_{ka}(0,z)-p_{ka}(e_i,z)|\\
&=Ce^{-ka_i}(1+\sum_{l=1}^\infty|\frac{(ka_i)^l}{l!}-\frac{(ka_i)^{l-1}}{(l-1)!}|).
\end{align*}
Letting $k\to\infty$, we get $f_n(x)=f_n(x+e_i)$, and $f_n$ is a constant function. 
\end{proof}

Let $A$ be a stable simple AF algebra not isomorphic to $\K$. 
Then we may assume $A=B\otimes \K$ with an infinite dimensional unital simple AF algebra $B$. 
We can express $B$ and $\K$ as $B=\varinjlim B_n$ with finite dimensional $B_n$ 
and $\K=\varinjlim M_n(\C)$. 
We may assume that $B_n\subset B$ is a unital inclusion. 
Then we have 
$$U_\to(A)=\varinjlim U_\to(B\otimes M_n(\C)).$$
We set $G_n=U_\to(B\otimes M_n(\C))$. 
For $n<m$, we denote by $\Phi_{n,m}$ the connecting map from $G_n$ to $G_m$ given by 
$\Phi_{n,m}(u)=u+(1_m-1_n)$ where $1_n$ is the unit of $B\otimes M_n(\C)$. 
We denote by $\Phi_{n,\infty}$ the inclusion map from $G_n$ to $U(A)$ given by 
$\Phi_{n,\infty}(u)=u+1-1_n$. 

Assume that $TW(A)$ is finite dimensional. 
Then $\ex T(B)$ is a finite set, and we denote $\ex T(B)=\{\tau_i\}_{i=1}^s$. 
Let $\Tr$ and $\Tr_n$ be the traces of $\K$ and $M_n(\C)$ respectively. 
We set 
$$\tau_{i,n}=\frac{1}{n}\tau_i\otimes \Tr_n.$$
Then $\ex T(B\otimes M_n(\C))=\{\tau_{i,n}\}_{i=1}^s$.

\begin{lemma} Let $A$ be a stable simple AF algebra not isomorphic to $\K$. 
If the affine space $TW(A)$ is finite dimensional, then 
$\chi_{\tau,\tau'}\in \ex \Char(U_\to(A))$ for any $\tau,\tau'\in TW(A)$. 
\end{lemma}

\begin{proof} We use the notation as above. 
Then there exist non-negative numbers $a_i$ and $b_i$ such that 
$\tau=\sum_{i=1}^s a_i\tau_i\otimes \Tr$ and 
$\tau'=\sum_{i=1}^sb_i\tau_i\otimes \Tr$. 
We set $I=\{i;\; a_i\neq 0\}$, $J=\{j;\; b_j\neq 0\}$.  

For any $u\in G_n$, we have 
\begin{align*}
\lefteqn{\chi_{\tau,\tau'}(\Phi_{n,\infty}(u))=\prod_{i\in I}e^{na_i\tau_{i,n}(u-1_n)}
\prod_{j\in J}e^{nb_j\tau_j(u^*-1_n)}} \\
 &=\prod_{i\in I}\big(\sum_{k=0}^\infty\frac{(na_i)^ke^{-na_i}}{k!}\tau_{i,n}(u)^k\big)
 \prod_{j\in J}\big(\sum_{l=0}^\infty\frac{(nb_j)^le^{-nb_j}}{l!}\overline{\tau_{j,n}(u)}^l\big)\\
&=\sum_{(x,y)\in \Z_{\geq 0}^I\times \Z_{\geq 0}^J}a_n(x)b_m(y)
\prod_{i\in I}\tau_{i,n}(u)^{x_i}\prod_{j\in J}\overline{\tau_{j,n}(u)}^{y_i}. 
\end{align*}
where 
$$a_n(x)=\prod_{i\in I}\frac{(na_i)^{x_i}e^{-na_i}}{{x_i}!},\quad 
b_n(y)=\prod_{j\in J}\frac{(nb_j)^{y_j}e^{-nb_j}}{{y_j}!}.$$

Let $\chi\in \Char(U_\to(A))$ be a character for which $c\chi_{\tau,\tau'}-\chi$ is positive definite 
with some constant $c>0$. 
Then there exist unique positive numbers $f_n(x,y)$ for $(x,y)\in \Z_{\geq 0}^{s_1+s_2}$ such that 
$0\leq f_n(x,y)\leq c$ and 
$$\chi(\Phi_{n,\infty}(u))=\sum_{(x,y)\in \Z_{\geq 0}^I\times \Z_{\geq 0}^J}f_n(x,y)a_n(x)b_n(y)
\prod_{i\in I}\tau_{i,n}(u)^{x_i}\prod_{j\in J}\overline{\tau_{j,n}(u)}^{y_j}.$$

For $n<m$, we have $\chi(\Phi_{n,\infty}(u))=\chi(\Phi_{m,\infty}\circ\Phi_{n,m}(u))$, and 
the right-hand side is  
$$\sum_{(z,w)\in \Z_{\geq 0}^I\times \Z_{\geq 0}^J}f_m(z,w)a_m(z)b_m(w)\prod_{i\in I}\tau_{i,m}(\Phi_{n,m}(u))^{z_i}
\prod_{j\in J}\overline{\tau_{j,m}(\Phi_{n,m}(u))}^{w_j}.$$
Since 
\begin{align*}
\lefteqn{\prod_{i\in I}\tau_{i,m}(\Phi_{n,m}(u))^{z_i}=\prod_{i\in I}(\frac{n\tau_{i,n}(u)+m-n}{m})^{z_i}} \\
 &=\prod_{i\in I}(\sum_{x_i=0}^{z_i}
\left(
\begin{array}{c}
z_i\\
x_i 
\end{array}
\right)
\frac{n^{x_i}(m-n)^{z_i-x_i}}{m^{z_i}}\tau_{i,n}(u)^{x_i})\\
&=\sum_{x\in \Z_{\geq 0}^I,\;x\leq z}\prod_{i\in I}
\left(
\begin{array}{c}
z_i\\
x_i 
\end{array}
\right)
\frac{n^{x_i}(m-n)^{z_i-x_i}}{m^{z_i}}\tau_{i,n}(u)^{x_i},
\end{align*}
where $x\leq z$ means $x_i\leq z_i$ for all $i\in I$, 
we obtain 
\begin{align*}
\lefteqn{\chi(\Phi_{m,\infty}\circ\Phi_{n,m}(u))
=\sum_{0\leq x\leq z,\;0\leq y\leq w}f_m(z,w)}\\
&\times \frac{a_m(z)}{a_n(x)}
\prod_{i\in I}
\left(
\begin{array}{c}
z_i\\
x_i 
\end{array}
\right)
\frac{n^{x_i}(m-n)^{z_i-x_i}}{m^{z_i}}  \frac{b_m(w)}{b_n(y)} 
 \prod_{j\in J}
\left(
\begin{array}{c}
w_j\\
y_j 
\end{array}
\right)
\frac{n^{y_j}(m-n)^{w_j-y_j}}{m^{w_j}}\\
&\times a_n(x)b_n(y)\prod_{i\in I}\tau_{i,n}(u)^{x_i}\prod_{j\in J}\overline{\tau_{j,n}(u)}^{y_j},
\end{align*}
where $x,z\in \Z_{\geq 0}^I$ and $y,w\in \Z_{\geq 0}^J$. 
Since 
$$\frac{a_m(z)}{a_n(x)}
\prod_{i\in I}
\left(
\begin{array}{c}
z_i\\
x_i 
\end{array}
\right)
\frac{n^{x_i}(m-n)^{z_i-x_i}}{m^{z_i}}
=\prod_{i\in I} 
\frac{a_i^{z_i-x_i}(m-n)^{z_i-x_i}e^{-(m-n)a_i}}{(z_i-x_i)!},
$$
we get 
\begin{align*}
\lefteqn{f_n(x,y)} \\
 &=\sum_{z\geq x,\;w\geq y}f_m(z,w) \prod_{i\in I} 
\frac{\{(m-n)a_i\}^{z_i-x_i}e^{-(m-n)a_i}}{(z_i-x_i)!}
\prod_{j\in J}
\frac{\{(m-n)b_j\}^{w_j-y_j}e^{-(m-n)b_j}}{(w_j-y_j)!}.
\end{align*}
Now Lemma \ref{tail boundary} implies that $\{f_n\}_{n=1}$ is a constant sequence consisting of a constant 
function, and $\chi=\chi_{\tau,\tau'}$. 
This shows $\chi_{\tau,\tau'}\in \ex\Char(U_\to(A))$. 
\end{proof}

\begin{proof}[Proof of Theorem \ref{main3}] 
Let $\chi\in \ex\Char(U_\to(A))$. 
Applying the ergodic method to $\varinjlim G_n$ with $G_n=U_\to(B\otimes M_n(\C))$, 
we get a sequence of characters $\chi_n\in \ex\Char(G_n)$ such that 
$\{\chi_m(\Phi_{n,m}(u))\}_{m=n}^\infty$ converges to $\chi(\Phi_{n,\infty}(u))$ uniformly on $G_n$ for every $n$, 
and $\chi_m\circ \Phi_{n,m}$ contains $\chi_n$ for any $n<m$.  
Theorem \ref{main1} shows that there exist $\varphi_n\in \Hom(K_0(A),\Z)$ and $p_{i,n},q_{i,n}\in \Z_{\geq 0}$ 
such that for any $u\in G_n$ we have
$$\chi_n(u)=\mathrm{det}_{\varphi_n}(u)
\prod_{i=1}^s\tau_{i,n}(u)^{p_{i,n}}\overline{\tau_{i,n}(u)}^{q_{i,n}}.$$
Moreover we have
\begin{align*}
\lefteqn{\chi_m(\Phi_{n,m}(u))} \\
 &=\mathrm{det}_{\varphi_{m}}(u)\prod_{i=1}^s(\frac{n\tau_{i,n}(u)+m-n}{m})^{p_{i,m}}
 (\frac{n\overline{\tau_{i,n}(u)}+m-n}{m})^{q_{i,m}}\\
 &=\mathrm{det}_{\varphi_{m}}(u)\prod_{i=1}^s(1+\frac{n\tau_{i,n}(u-1_n)}{m})^{p_{i,m}}
 (1+\frac{n\tau_{i,n}(u^*-1_n)}{m})^{q_{i,m}}. 
\end{align*}
Thus $\{\varphi_m\}_{m=1}^\infty$ is a constant sequence, 
say $\varphi$, and $\{p_{i,m}\}_{m=1}^\infty$ and $\{q_{i,m}\}_{m=1}^\infty$ are increasing. 

We claim that $\{p_{i,m}/m\}_{m=1}^\infty$ and $\{q_{i,m}/m\}_{m=1}^\infty$ are bounded. 
For $u\in G_n$, let $z_i=\tau_{i,n}(u-1_n)$. 
Then $\Re z_i\leq 0$, and $\Re z_i=0$ if and only if $u=1_n$. 
Since $1+t\leq e^{t}$ holds for any $t\in \R$, we have 
$$|(1+\frac{nz_i}{m})^{p_{i,m}}|=(1+\frac{2n\Re z_i+|n^2z_i|^2/m}{m})^{\frac{p_{i,m}}{2}}\leq 
e^{(n\Re z_i+|n^2z_i|^2/2m)\frac{p_{i,m}}{m}}.$$
Thus if there existed $i$ such that either $\{p_{i,m}/m\}_{m=1}^\infty$ or $\{q_{i,m}/m\}_{m=1}^\infty$ is 
unbounded, we would have 
$$\chi(v)=\left\{
\begin{array}{ll}
1 , &\quad v=1 \\
0 , &\quad v\neq 1
\end{array}
\right.,$$
which is not continuous. 
Thus  $\{p_{i,m}/m\}_{m=1}^\infty$ and $\{q_{i,m}/m\}_{m=1}^\infty$ are bounded. 

Choosing an appropriate subsequence, we may assume that the following limits exist: 
$$a_i=\lim_{k\to \infty}\frac{p_{i,m_k}}{m_k},\quad b_i=\lim_{k\to \infty}\frac{q_{i,m_k}}{m_k},\quad 
1\leq \forall i\leq s.$$
Thus 
\begin{align*}
\lefteqn{\chi(\Phi_{n,\infty}(u))=\lim_{m\to\infty}\chi_m(\Phi_{n,m}(u))} \\
 &=\lim_{k\to\infty}\mathrm{det}_{\varphi}(u)
 \prod_{i=1}^s(1+\frac{n\tau_{i,n}(u-1_n)}{m_k})^{p_{i,m_k}}
 (1+\frac{n\tau_{i,n}(u^*-1_n)}{m_k})^{q_{i,m_k}}\\
 &=\mathrm{\det}_{\varphi}(u)\prod_{i=1}^s e^{a_in\tau_{i,n}(u-1_n)}e^{b_in\tau_{i,n}(u^*-1_n)}.
\end{align*}
Letting 
$$\tau=\sum_{i=1}^sa_i\tau_i\otimes \Tr,\quad \tau'=\sum_{i=1}^sb_i\tau_i\otimes \Tr,$$
we get $\chi=\det_\varphi \chi_{\tau,\tau'}$, which finishes the proof. 
\end{proof}

It is natural to ask whether we can drop the condition $\dim TW(A)<\infty$ in Theorem \ref{main3}. 

\begin{con} Let $A$ be a stable simple AF algebra not isomorphic to $\K$. 
Then 
$$\ex \Char(U_\to(A))=\{\mathrm{det}_\varphi\chi_{\tau,\tau'};\; \tau,\tau'\in TW(A),\; 
\varphi\in \Hom(K_0(A),\Z)\}.$$
\end{con}

\section{Type II factors}
Our purpose in this final section is to prove Theorem \ref{II1} and Theorem \ref{IIinfty}. 

\begin{proof}[Proof of Theorem \ref{II1}] 
We first show the statement for the hyperfinite II$_1$ factor $\cR_0$.
We can choose a dense $C^*$-subalgebra $A\subset \cR_0$ isomorphic to the CAR algebra. 
Since $U(A)$ is dense in $U(\cR_0)$ and the embedding map of $U(A)$ with the norm topology 
into $U(\cR_0)$ with the strong operator topology is continuous, the statement follows from 
Theorem \ref{main1} and Corollary \ref{decomposition}. 

Let $R$ be a general II$_1$ factor now, and let $\tau$ be the unique tracial state on $R$. 
Then there exists a unital embedding of the hyperfinite II$_1$ factor $\cR_0$ into $R$. 
We show that every $\chi \in \Char(U(R))$ is uniquely decomposed as 
$$\chi=\sum_{p,q\in \Z_{\geq 0}}c_{p,q}\tau^p\overline{\tau}^q,\quad c_{p,q}\geq 0,$$
which will finish the proof. 
Indeed, the restriction of $\chi$ to $U(\cR_0)$ is uniquely decomposed as  
$$\chi|_{U(\cR_0)}=\sum_{p,q\in \Z_{\geq 0}}c_{p,q}\tau_{\cR_0}^p\overline{\tau_{\cR_0}}^q,\quad c_{p,q}\geq 0,$$
where $\tau_{\cR_0}$ is the restriction of $\tau$ to $\cR_0$. 
As in the proof of Lemma \ref{invaraince}, we can show that for any $u\in U(R)$ there exist sequences of unitaries 
$u_n\in U(\cR_0)$ and $w_n\in U(R)$ satisfying $\|w_nuw_n^*-u_n\|\to 0$, and we get
$$\chi(u)=\lim_{n\to\infty}\chi(u_n)=\lim_{n\to\infty}\sum_{p,q\geq 0}c_{p,q}\tau(u_n)^p\overline{\tau(u_n)}^q=
\sum_{p,q\geq 0}c_{p,q}\tau(u)^p\overline{\tau(u)}^q.$$
Since the decomposition is unique on $U(\cR_0)$, it is unique on $U(R)$ too. 
\end{proof}

Let $M$ be a type II$_\infty$ factor with separable predual, and let $\tau_\infty$ be a 
normal semifinite trace on $M$. 
Then there exists a II$_1$ factor $R$ with a unique tracial state $\tau$ satisfying  
$M=R\otimes B(\ell^2)$ and $\tau_\infty=\tau\otimes \Tr$. 
Let $\{e_{ij}\}_{i,j\in \N}$ be the canonical system of matrix units of $B(\ell^2)$, 
and let $e_n=\sum_{i=1}^n 1\otimes e_{ii}\in M$. 
Then $\tau_\infty(e_n)=n$ and $\{e_n\}_{n=1}^\infty$ converges to $1$ in the strong operator topology. 
We set $G_n=U(e_nMe_n)$, which has the strong operator topology, and embed it into $U(M)_1$ by $u\mapsto u+1-e_n$. 
We regard $G=\varinjlim G_n$ as a subgroup of $U(M)_1$.

\begin{lemma} The group $G$ is dense in $U(M)_1$. 
\end{lemma}

\begin{proof} Let $u\in U(M)_1$, and let $\varepsilon>0$. 
Then 
\begin{align*}
\lefteqn{\|(u-1)(1-e_n)\|_1=\||u-1|(1-e_n)\|_1} \\ 
&\leq \||u-1|^{1/2}\|_2\||u-1|^{1/2}(1-e_n)\|_2\\
&=\|u-1\|_1^{1/2}\sqrt{\tau_\infty(|u-1|)-\tau_\infty(|u-1|^{1/2}e_n|u-1|^{1/2})},
\end{align*} 
which converges to 0 as $n$ tends to $\infty$. 
A similar estimate holds for $\|(1-e_n)(u-1)\|_1$, and we can choose $n$ satisfying 
$\|(u-1)(1-e_n)\|_1<\varepsilon$, and $\|(1-e_n)(u-1)\|_1<\varepsilon$. 
Thus 
\begin{align*}
\lefteqn{\|e_nue_n+1-e_n-u\|_1}\\
&=\|(1-e_n)(1-u)(1-e_n)-(1-e_n)ue_n-e_nu(1-e_n)\|_1\\
&\leq \|(1-e_n)(1-u)(1-e_n)\|_1+\|(1-e_n)(u-1)e_n\|_1+\|e_n(u-1)(1-e_n)\|_1\\
&\leq 2\|(u-1)(1-e_n)\|_1+\|(1-e_n)(u-1)\|_1<3\varepsilon. 
\end{align*}

We would like to approximate $e_nue_n$ by a unitary in $e_nMe_n$. 
Since $1-t\leq (1-t^2)^{1/2}$ holds for $0\leq t\leq 1$, \begin{align*}
\lefteqn{\tau_\infty(e_n-|e_nue_n|)\leq \tau_{\infty}((e_n-e_nu^*e_nue_n)^{1/2})
} \\
 &=\tau_{\infty}(\{e_nu^*(1-e_n)ue_n\}^{1/2})=\tau_{\infty}(\{e_n(u^*-1)(1-e_n)(u-1)e_n\}^{1/2})\\
 &=\|(1-e_n)(u-1)e_n\|_1\leq \|(1-e_n)(u-1)\|_1<\varepsilon. 
\end{align*}
Let $e_nue_n=v_n|e_nue_n|$ be the polar decomposition. 
Since $e_nMe_n$ is finite, we can choose a partial isometry $w_n\in e_nMe_n$ such that 
$w_n^*w_n=e_n-v_n^*v_n$ and $w_n+v_n$ is a unitary in $U(e_nMe_n)$.  
Now we have 
\begin{align*}
\lefteqn{\|w_n+v_n+1-e_n-u\|_1\leq \|w_n+v_n-e_nue_n\|_1+\|e_nue_n+1-e_n-u\|_1} \\
 &<3\varepsilon +\|w_n\|_1+\|v_n-v_n|e_nue_n|\|_1
 =3\varepsilon +\tau_\infty(e_n-v_n^*v_n)+\tau_\infty(v_n^*v_n-|e_nue_n|)\\
 &=3\varepsilon +\tau_\infty(e_n-|e_nue_n|)<4\varepsilon,
\end{align*}
which shows that $G$ is dense in $U(M)_1$. 
\end{proof}

\begin{proof}[Proof of Theorem \ref{IIinfty}] 
We equip $G$ with the inductive limit topology. 
Then essentially the same argument as in the proof of Theorem \ref{main3} shows 
$$\ex\Char(G)=\{\chi_{a,b};\; a,b\geq 0\}.$$
Since the embedding map from $G$ into $U(M)_1$ is continuous and $G$ is dense in $U(M)_1$, 
we get (1). 

For $1\leq q<p$ and $u,v\in U(M)_{q}$, we have 
$$\|u-v\|_p=\tau_\infty(|u-v|^p)^{1/p}\leq (\||u-v|^{p-q}\|\tau(|u-v|^q))^{1/p}\leq 2^{1-q/p}\|u-v\|_q^{q/p},$$
which implies that the embedding of $U(M)_q$ into $U(M)_p$ is continuous. 
It is easy to show that $U(M)_q$ is dense in $U(M)_p$. 
Thus we can repeat a similar argument as above replacing $G$ with $U(M)_1$ and $U(M)_2$ 
for (2) and (3) respectively. 

Let $1<p\leq 2$. 
For $u\in U(M)_1$, we have $\chi_{a,b}(u)=\chi_c(u)\chi_{a-c,b-c}(u)$ with $c=\min\{a,b\}$. 
Since $\chi_c$ is continuous on $U(M)_p$, the character $\chi_{a,b}$ continuously extends to 
$U(M)_p$ if and only if $a-c=b-c=0$, and we get (2). 

Let $2<p$. 
Then $\chi_a$ continuously extends to $U(M)_p$ if and only if $a=0$, and we get (3).  
\end{proof}


\end{document}